\documentclass[a4paper,10pt]{article}

\usepackage{a4wide}

\usepackage[latin1]{inputenc}
\usepackage{subfigure}

\usepackage{amsmath, amssymb, amsthm}
\usepackage{tikz}
\usepackage{enumitem}
\usepackage{pdflscape}
\usepackage{mathtools}
\usepackage{hyperref}

\newtheorem{theorem}{Theorem}
\newtheorem{definition}[theorem]{Definition}
\newtheorem{observation}[theorem]{Observation}
\newtheorem{remark}[theorem]{Remark}
\newtheorem{conjecture}[theorem]{Conjecture}

\newcommand{\commonpat}{2\underbracket[.5pt][1pt]{41}3}
\newcommand{\baxpat}{3\underbracket[.5pt][1pt]{14}2}
\newcommand{\TCat}{\mathcal{T}_{\textrm{Cat}}}
\newcommand{\TBax}{\mathcal{T}_{\textrm{Bax}}}
\newcommand{\TSch}{\mathcal{T}_{\textrm{Sch}}}

\DeclareMathOperator{\Bax}{Bax}

\newcommand{\myS}{\mathcal{S}}

\author{Nicholas R. Beaton, Mathilde Bouvel, Veronica Guerrini and Simone Rinaldi}

\title{Slicings of parallelogram polyominoes: \\ Catalan, Schr\"oder, Baxter, and other sequences}

\begin{document}

\maketitle

\begin{abstract} 
We provide a new succession rule (\emph{i.e.}~generating tree) associated with Schr\"oder numbers, 
that interpolates between the known succession rules for Catalan and Baxter numbers. 
We define Schr\"oder and Baxter generalizations of parallelogram polyominoes, called \emph{slicings}, 
which grow according to these succession rules. 
In passing, we also exhibit Schr\"oder subclasses of Baxter classes, namely 
a Schr\"oder subset of triples of non-intersecting lattice paths, 
a new Schr\"oder subset of Baxter permutations, 
and a new Schr\"oder subset of mosaic floorplans. 
Finally, we define two families of subclasses of Baxter slicings: the $m$-skinny slicings and the $m$-row-restricted slicings, for $m \in \mathbb{N}$. 
Using functional equations and the kernel method, their generating functions are computed in some special cases, 
and we conjecture that they are algebraic for any $m$.
\end{abstract}


\section{Introduction}
\label{sec:intro}

The sequence of Catalan numbers (\textsc{a000108} in~\cite{OEIS}) is arguably the most well-known combinatorial sequence. 
It is known to enumerate dozens of families of combinatorial objects, including
Dyck paths, 
parallelogram polyominoes, 
and $\tau$-avoiding permutations, for any permutation $\tau$ of size $3$. 
In this paper, we are interested in Catalan numbers as well as in two larger combinatorial sequences: the Schr\"oder and Baxter numbers. 

Baxter numbers (sequence \textsc{a001181}) were first introduced in~\cite{CGHK78}, 
where it is shown that they count Baxter permutations. 
They also enumerate numerous families of combinatorial objects, 
and their study has attracted significant attention, see for instance~\cite{mbm,ffno}. 
Many such Baxter families can be immediately seen to contain a Catalan subfamily. 
For instance, the set of triples of non-intersecting lattice paths (NILPs) contains all pairs of NILPs (which are in essence parallelogram polyominoes, see Figure~\ref{fig:PP}(a) and the blue and red paths of Figure~\ref{fig:PP}(c)); 
and Baxter permutations, defined by the avoidance of the vincular\footnote{Note that we do not represent vincular patterns with dashes, 
as it was done originally. We prefer the more modern and more coherent notation that indicates by a symbol $\underbracket[.5pt][1pt]{~\ ~}$ 
the elements of the pattern that are required to be adjacent in an occurrence.} 
patterns $2\underbracket[.5pt][1pt]{41}3$ and $3\underbracket[.5pt][1pt]{14}2$, 
include $\tau$-avoiding permutations, for any $\tau \in \{132,213,231,312\}$. 

On the other hand, the (large) Schr\"oder numbers (sequence \textsc{a006318}) seem to be a bit less popular. 
They also form a sequence point-wise larger than the Catalan sequence, 
and it is additionally point-wise smaller than the Baxter sequence. 
This is most easily seen by considering permutations, where the Schr\"oder numbers count the separable permutations~\cite{ShaSte,West}, 
defined by the avoidance of $2413$ and $3142$. 

The first purpose of this article is to explain and illustrate the inclusions ``\emph{Catalan in Schr\"oder in Baxter}''. 
Although these inclusions are obvious on pattern-avoiding permutations, they remain quite obscure on other objects. 
Indeed, looking at several combinatorial objects, it appears that 
the permutation example is a little miracle, and that the unclarity of these inclusions is rather the rule here. 
To give only a few examples, consider for instance lattice paths: 
the Dyck paths generalize to Schr\"oder paths (by allowing an additional flat step of length $2$), but have, to our knowledge, no natural Baxter analogue; 
on the contrary, pairs of NILPs are counted by Catalan, whereas triples of NILPs are counted by Baxter, leaving Schr\"oder aside. 
Or, consider another well-known Catalan family: that of binary trees. 
There are Schr\"oder and Baxter objects generalizing binary trees 
(Schr\"oder trees, with an additional sign on the root on one hand, 
or pairs of twin binary trees on the other), 
but they have apparently nothing in common. 

As these examples illustrate, the Baxter and Schr\"oder generalizations of Catalan objects are often independent and are not easily reconciled. 
This fact is also visible at a more abstract level, \emph{i.e.}~without referring to specific combinatorial families: 
by considering the generating trees (with their corresponding succession rules) associated with these sequences 
(we will review the basics of generating trees in Section~\ref{sec:Catalan}). 
As we demonstrate in this work, for the known generating trees associated with the Schr\"oder and Baxter numbers, 
when they can be seen as generalizations of the generating tree of Catalan numbers, 
then these two generalizations go in two opposite directions. 
Our first contribution is to provide a continuum from Catalan to Baxter via Schr\"oder 
that is visible at the abstract level of succession rules. 
In this paper, as well as in our recent works~\cite{BijectionNick,SemiAndStrong}, 
we consider several generating trees and their associated succession rules, 
and we focus on succession rules that \emph{generalize}, or conversely \emph{specialize}, well-known succession rules. 
Although this can be understood at a rather informal level (and this is actually how we originally worked), we propose a formalization of this idea of generalizing (resp. specializing) a succession rule in Section~\ref{sec:gen-spec}.

We will focus mostly on generalizations of parallelogram polyominoes, which we call \emph{slicings} of parallelogram polyominoes. 
In particular, Section~\ref{sec:Baxter} defines our Baxter slicings (also showing their tight connection with triples of NILPs). 
These new objects allow us to see that the usual Baxter succession rule does nothing but symmetrize the Catalan succession rule. 
Then in Section~\ref{sec:Schroeder}, we introduce a new succession rule associated with Schr\"oder numbers 
that interpolates between the Catalan and Baxter rules of Sections~\ref{sec:Catalan} and~\ref{sec:Baxter}. 
Letting our slicings grow with this rule allows us to define the family of Schr\"oder slicings. 
From there, the final sections go in different directions. 

Section~\ref{sec:other_Schroeder} presents other Schr\"oder subclasses of Baxter classes, 
obtained via our new Schr\"oder succession rule. This includes triples of NILPs, permutations and mosaic floorplans. 
Note that Schr\"oder subclasses of Baxter permutations and of mosaic floorplans already appear in the literature, 
like the separable permutations~\cite[for instance]{West} and the slicing floorplans~\cite{Yao}: 
our Schr\"oder subclasses are different from these. 
For triples of NILPs on the contrary, we are not aware of any known Schr\"oder subclass.

In Section~\ref{sec:extensions}, we introduce more intermediate classes between Catalan and Baxter, 
refining our new Schr\"oder succession rule with an integer parameter $m$ that may vary. 
This results in two families of subclasses of Baxter slicings: the $m$-skinny slicings and the $m$-row-restricted slicings. 
Section~\ref{sec:gen_funcs} is interested in the generating functions for these subclasses. 
First, the succession rules for $m$-skinny slicings and $m$-row-restricted slicings are translated into systems of equations for their generating functions. 
For the first values of $m$, these systems can be solved using the kernel method, showing an intriguing enumerative coincidence. 
Although we were not able to solve these systems for general $m$, we present a method to reach this goal, 
which fails only because we were not able to prove that the power series solutions of a certain equation are linearly independent. 
Note that this property is indeed verified for a few more values of $m$,  as Table~\ref{table:conj} (page~\pageref{table:conj}) sums up, solving a few more cases of the enumeration of $m$-skinny slicings and $m$-row-restricted slicings. 
In view of our method, we offer the conjecture that the generating functions for $m$-skinny slicings and $m$-row-restricted slicings are algebraic, for all $m$. 

\vspace{-0.1cm}

\section{Parallelogram polyominoes and the generating tree for Catalan numbers
\label{sec:Catalan}}

There are many ways of defining (or characterizing) parallelogram polyominoes in the literature, 
and we only give one that fits our needs. 

\begin{definition}
A \emph{parallelogram polyomino} $P$ (see an example in Figure~\ref{fig:PP}(a)) is an (edge-)connected set of unit cells in the Cartesian plane 
that is the interior of a contour defined by two paths which are composed of $(0,1)$ and $(1,0)$ steps 
and which never meet except at their beginning and end. 
Denoting $(k,\ell)$ the dimension of the minimal bounding rectangle of $P$ 
($k$ being its width and $\ell$ its height), 
the semi-perimeter of $P$ is $k+\ell$, and the size of $P$ is $k+\ell -1$. 
\end{definition}

\begin{figure}[ht]
\begin{center}
\includegraphics[scale=0.7]{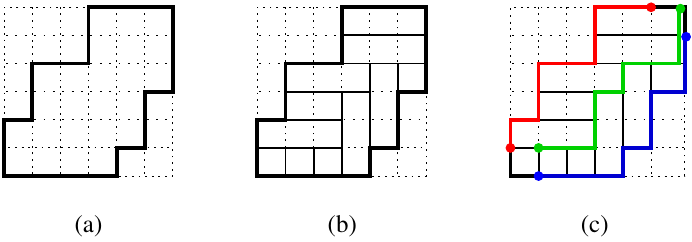}
\end{center}
\caption{(a) A parallelogram polyomino $P$ of size $11$, 
(b) a Baxter slicing of shape $P$, and (c) the triple of NILPs associated with it. 
\label{fig:PP}} 
\end{figure}

We start by reviewing generating trees~\cite{GFGT,Eco,West}, 
and in particular the generating tree for Catalan numbers associated with parallelogram polyominoes. 

A \emph{generating tree} for a combinatorial class $\mathcal{C}$ is an infinite rooted tree, whose vertices are the objects of $\mathcal{C}$, 
each appearing exactly once in the tree, 
and such that objects of size $n$ are at level $n$ in the tree
(with the convention that the root is at level $1$, and is labeled by the only object of size $1$ in $\mathcal{C}$). 
The children of some object $c \in \mathcal{C}$ are obtained by adding an \emph{atom} (\emph{i.e.}~a piece of object that makes its size increase by $1$) to $c$. 
Of course, since every object should appear only once in the tree, not all additions are possible. 
We should ensure the unique appearance property by considering only additions that follow some restricted rules. 
We will call the \emph{growth} of $\mathcal{C}$ the process of adding atoms following these prescribed rules. 

A generating tree of parallelogram polyominoes was described in~\cite{Eco}, by means of a so-called \emph{ECO operator}, 
and its first levels are illustrated in Figure~\ref{fig:growth_PP}. 
The atoms that may be inserted are rightmost columns 
(of any possible height -- \emph{i.e.}, number of cells -- from $1$ to the height of the current rightmost column), 
and topmost rows of width $1$. 
Note that the restriction on the width of the new row added is here only to ensure that no polyomino is produced several times. 
Note also that the mirror image of this growth rule, which allows rows of any admissible width but columns of height $1$ only, 
also describes a generating tree for parallelogram polyominoes, which is isomorphic to the first one.

\begin{figure}[ht]
\begin{center}
\includegraphics[scale=0.6]{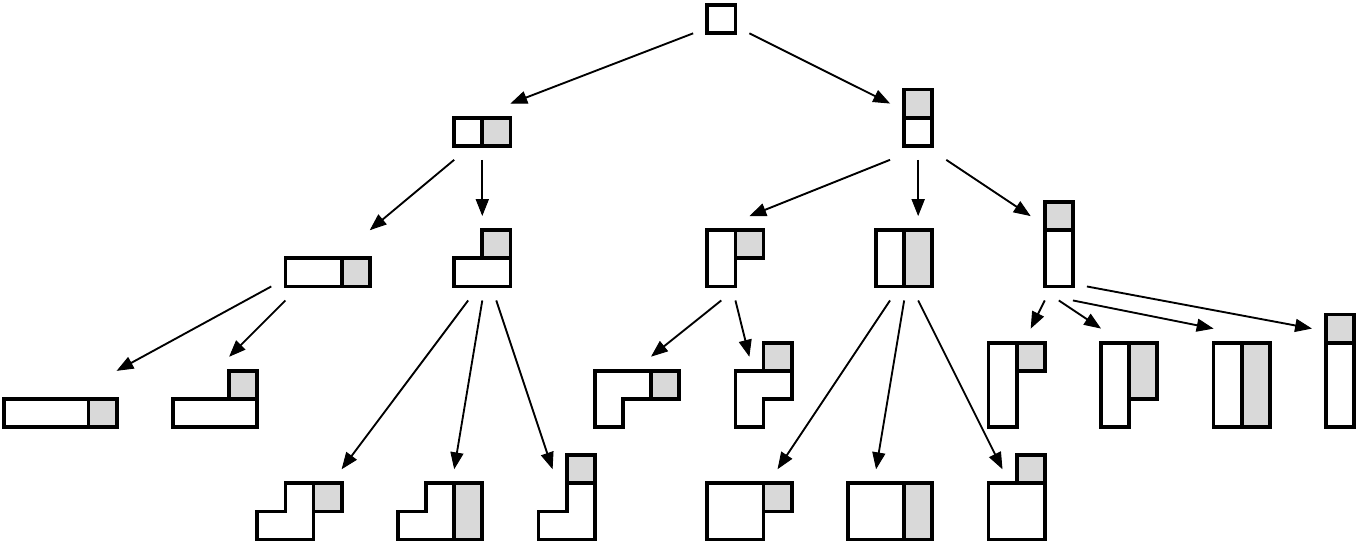}
\end{center}
\caption{The first levels of the generating tree of parallelogram polyominoes. \label{fig:growth_PP}}
\end{figure}

All that matters to us is the \emph{shape} of a generating tree, forgetting the combinatorial objects on the vertices. 
In what follows, we will use the phrase ``generating tree'' to denote this shape only, 
referring instead to ``full generating trees'' when the nodes are carrying combinatorial objects. 

Generating trees become substantially useful if they can be described in an abstract way, without referring to the details of the combinatorial objects. 
More precisely, for a combinatorial class $\mathcal{C}$, 
assuming that there is a statistic on the objects of $\mathcal{C}$ 
whose value determines the number of children in the full generating tree, 
then the (shape of the) generating tree depends only on how the value of the statistic evolves from an object to its children. 
When such a statistic exists, we give \emph{labels} to the objects of $\mathcal{C}$, 
which indicate the value of the statistic. 
The associated \emph{succession rule} is then given by 
the label of the root and, for any label $k$, the labels of the children of an object labeled by $k$. 
A succession rule characterizes completely a generating tree. 

In the case of parallelogram polyominoes, 
the number of children is determined by the height of the rightmost column (namely, it is this height $+1$), 
and it is easy to trace the height of the rightmost column as a polyomino grows in size. 
It follows that the generating tree of parallelogram polyominoes described above is 
completely determined by the following succession rule: 
\begin{equation}
\textrm{root labeled } (1) \textrm{ \quad and \quad } (k) \rightsquigarrow (1), (2), \ldots , (k), (k + 1). \tag{\textrm{Cat}} \label{rule:Catalan}
\end{equation}
We will denote this generating tree by $\TCat$ and its first levels are represented in Figure~\ref{fig:GeneratingTrees} (page~\pageref{fig:GeneratingTrees}). 

Note that, given a succession rule and its subsequent generating tree, 
we can associate with it an enumeration sequence, whose $n$-th term $c_n$ is the number of vertices in the tree at level $n$. 
Of course, $(c_n)$ is the enumeration sequence of any combinatorial class that has a (full) generating tree encoded by the given succession rule. 
But our point, which will be essential later on, is that the sequence may also be associated directly with the generating tree, 
without reference to any combinatorial class. 
In our example, it follows that rule \eqref{rule:Catalan} (and the corresponding tree $\TCat$) is associated with the Catalan numbers, hence its name. 

\section{Baxter slicings}
\label{sec:Baxter}

\subsection{A Baxter succession rule generalizing the Catalan rule}
\label{subsec:Cat_in_Bax}

There are several succession rules associated with Baxter numbers~\cite{BM,BouvelGuibert,Mishna2,Mishna1}. 
We will be interested in one of these rules only which, in addition to being the most well-known, 
is the one that generalizes the rule for Catalan numbers in the most natural way. The rule is: 

\begin{equation}
\textrm{root labeled } (1,1) \textrm{ \quad and \quad } (h,k) \rightsquigarrow \begin{cases}
                        (1, k + 1), (2, k + 1), \ldots , (h, k + 1), \\
                        (h + 1, 1), (h + 1, 2), \ldots,  (h + 1, k).
                       \end{cases} \tag{\textrm{Bax}} \label{rule:BaxterMBM}
\end{equation}

We denote by $\TBax$ the generating tree associated with this rule, and illustrate it in Figure~\ref{fig:GeneratingTrees} (page~\pageref{fig:GeneratingTrees}). 
A proof that it corresponds to Baxter numbers can be found in~\cite{BM,Gire}, 
where it is proved that Baxter permutations grow according to rule~\eqref{rule:BaxterMBM}. 
Recall that Baxter permutations are those avoiding the vincular patterns $2\underbracket[.5pt][1pt]{41}3$ and $3\underbracket[.5pt][1pt]{14}2$, 
\emph{i.e.}~permutations $\sigma$ such that no subsequence $\sigma_i \sigma_j \sigma_{j+1} \sigma_k$ satisfies  
$\sigma_{j+1} < \sigma_i < \sigma_k < \sigma_j$ or $\sigma_{j} < \sigma_k < \sigma_i < \sigma_{j+1}$.
% From~\cite{BM,Gire}, the growth of Baxter permutations according to rule~\eqref{rule:BaxterMBM} 
% consists, for any Baxter permutation $\sigma$, in inserting a new maximum element either immediately to the left of a left-to-right maximum of $\sigma$, 
% or immediately to the right of a right-to-left maximum of $\sigma$. 
% The label $(h,k)$ of a permutation records the number of its left-to-right maxima (for $h$) and right-to-left maxima (for $k$). 

It is easily seen, however rarely noticed, that rule~\eqref{rule:BaxterMBM} generalizes rule~\eqref{rule:Catalan}. 
Indeed, the production of label $(h,k)$ in rule~\eqref{rule:BaxterMBM} includes 
labels $(h+1,i)$ for $1 \leq i \leq k$ (in the second row of the production) and label $(1,k+1)$ (among others, in the first row of the production).
Keeping track only of the second element of these labels gives back the Catalan rule~\eqref{rule:Catalan}. 
(Observe that, comparing the growth of Baxter slicings -- defined later -- with that of parallelogram polyominoes, 
it is natural to keep the label $(1,k+1)$ in the first row, rather than, for instance, $(h,k+1)$.)
%Moreover, for another subset of the labels produced, the same holds keeping track of the first element only.

In some sense, rule~\eqref{rule:BaxterMBM} is just the symmetric version of rule~\eqref{rule:Catalan}. 
This is easy to see by considering the growth of parallelogram polyominoes according to rule~\eqref{rule:Catalan}. 
As we have seen, with rule~\eqref{rule:Catalan}, a rightmost column may be added, of all possible heights; 
but only a topmost row of width $1$ is allowed. 
But the symmetric variant of this rule, 
allowing addition of a topmost row of all possible widths, and of a rightmost column of height $1$, 
also works. 
So we can think of rule~\eqref{rule:BaxterMBM} as generating parallelogram polyominoes symmetrically, 
allowing at the same time the insertion of a rightmost column of any possible height, or of a topmost row of any possible width. 
Of course, this process generates the parallelogram polyominoes ambiguously.

\subsection{Definition and growth of Baxter slicings}

Our remark that rule~\eqref{rule:BaxterMBM} generates parallelogram polyominoes symmetrically but ambiguously 
motivates the definition of new combinatorial objects that generalize parallelogram polyominoes and grow unambiguously according to rule~\eqref{rule:BaxterMBM}. 
From the discussion above, the natural generalization is to let parallelogram polyominoes grow according to rule~\eqref{rule:BaxterMBM} as we explain, 
but to record the ``building history'' of the polyomino, that is, which columns and rows where added by the growth process. 
The objects obtained are parallelogram polyominoes whose interior is divided into blocks, of width or height $1$. 
We call these objects \emph{Baxter slicings of parallelogram polyominoes}, or \emph{Baxter slicings} for short. 

\begin{definition}
A \emph{Baxter slicing} (see an example in Figure~\ref{fig:PP}(b)) of size $n$ is a parallelogram polyomino $P$ of size $n$ whose interior is recursively divided into $n$ blocks as follows. 
\begin{itemize}
 \item The first block is the topmost row or the rightmost column of $P$ -- such blocks are called \emph{horizontal} and \emph{vertical} blocks, respectively. 
 \item The first block may be the topmost row (resp.~rightmost column) of $P$ only if the removal of the topmost row (resp.~rightmost column) of $P$ makes its semi-perimeter decrease by exactly $1$. 
 (Note that at least one of these two statements holds.)
 \item The other $n-1$ blocks form a Baxter slicing of the parallelogram polyomino of size $n-1$ obtained by deletion of the topmost row (resp.~rightmost column) of $P$. 
\end{itemize}
\label{dfn:Baxter_slicing}
\end{definition}

In the second item above, note that the condition that the semi-perimeter decreases by exactly $1$ 
is equivalent to the fact that 
the topmost row and the row below it end in the same rightmost column 
(resp. the rightmost column and the column to its left ends in the same top row). 

\begin{theorem}
Baxter slicings grow according to rule~\eqref{rule:BaxterMBM} and are enumerated by Baxter numbers. 
\label{thm:growth_Baxter_slicings}
\end{theorem}

\begin{proof}
It is clear that Baxter slicings grow according to rule~\eqref{rule:BaxterMBM}: 
a Baxter slicing has label $(h,k)$ when the topmost row has width $h$ and the rightmost column has height $k$, 
and the productions of label $(h,k)$ are immediately seen to correspond to the Baxter slicings obtained by adding 
a new horizontal block in a new topmost row, of any width between $1$ and $h$, 
or a new vertical block in a new rightmost column, of any height between $1$ and $k$. 
As a consequence, Baxter slicings are enumerated by Baxter numbers.
\end{proof}

An example of the growth of Baxter slicings according to rule~\eqref{rule:BaxterMBM} is shown in Figure~\ref{fig:Baxter_growth}. 

\begin{figure}[ht]
\begin{center}
\includegraphics[scale=0.6]{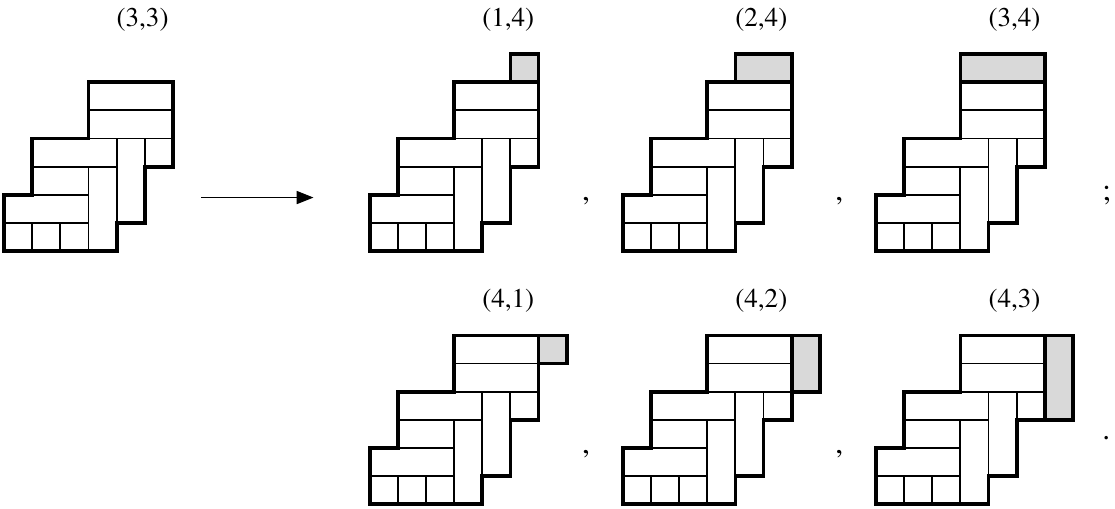}
\end{center}
\caption{The growth of Baxter slicings following rule~\eqref{rule:BaxterMBM}. \label{fig:Baxter_growth}}
\end{figure}

\subsection{Bijection with triples of NILPs}

Among the combinatorial families enumerated by Baxter numbers, one can be seen to be in bijection with Baxter slicings in a very simple way, 
namely, the triples of NILPs. 

\begin{definition}
A \emph{path} of size $n$ is a sequence of North ($N = (0,1)$) and East ($E=(1,0)$) steps, containing $n-1$ steps in total. 
Given three paths $u$, $m$, and $d$ of the same size $n$, 
all containing the same number of $E$ (and $N$) steps, 
$(u,m,d)$ is a \emph{triple of non-intersecting lattice paths} (for short, triple of NILPs) of size $n$ 
when the embeddings of $u$, $m$ and $d$ in the plane never meet, 
with $u$ (resp.~$m$, resp.~$d$) starting at the point of coordinates $(0,2)$ (resp.~$(1,1)$, resp.~$(2,0)$). 
\label{dfn:NILP}
\end{definition}
% 
% Note that, denoting $k$ the number of $E$ steps of each of the paths $u$, $m$ and $d$, 
% $u$ (resp.~$m$, resp.~$d$) ends in the point of coordinates $(k,n-k+1)$ (resp.~$(k+1,n-k)$, resp.~$(k+2,n-1-k)$). 

\begin{theorem}
The following construction, illustrated in Figure~\ref{fig:PP}(c), provides a size-preserving bijection between Baxter slicings and triples of NILPs:
\vspace{.2em}
\setlist{nolistsep}
\begin{description}[labelindent=.6cm]
\item[] Consider a Baxter slicing of a parallelogram polyomino $P$, whose bottom-left corner is assumed to be placed at coordinates $(0,0)$. 
Define the paths 
\begin{description}[noitemsep,labelindent=.5cm]
\item[-] $u$ corresponding to the upper border of $P$, except the first and last steps,
\item[-] $d$ corresponding to the lower border of $P$, except the first and last steps,
\item[-] and $m$ going from $(1,1)$ to the top-right corner of $P$, following the lower border of every horizontal block of the slicing, 
and the left border of every vertical block,
\end{description}
and associate the triple $(u,m,d)$ to the original Baxter slicing.
\end{description}
\label{thm:bijection_Baxter_slicings_NILPs}
\end{theorem}

\begin{proof}
Consider a Baxter slicing of a parallelogram polyomino $P$ of size $n$, and define $u,m$ and $d$ as above. 
Shifting by one the path $u$ (resp.~$d$) upwards (resp.~rightwards) so that the starting point is at $(0,2)$ (resp.~$(2,0)$), 
we want to prove $(u,m,d)$ is a triple of NILPs of size $n$. Note that by construction each step of the path $m$ is inside or on the border of the polyomino $P$; 
this immediately ensures the non-intersecting property. Moreover, by construction all paths $u,m$ and $d$ have $n-1$ steps, where $n+1$ denotes the semi-perimeter of $P$. 
Finally, we easily check that $u,m$ and $d$ have the same number of $E$ and $N$ steps. 
This follows immediately comparing the coordinates of the ending points of these paths. 
But it is also helpful (in the description of the inverse below or for the proof of Theorem~\ref{thm:bijection_to_Schroeder_NILPs}) to see it as we now explain. 
Since the path $m$ separates the horizontal blocks, which remain above it, 
from the vertical ones, which remain below it, each step of this path is either the right edge of a horizontal block or the upper edge of a vertical block. 
Then, the paths $u$ and $m$ have the same number of $N$ steps, 
as each $N$ step of the path $u$ is the left edge of a horizontal block.
Similarly, the paths $d$ and $m$ have the same number of $E$ steps, $E$ steps in $d$
corresponding to lower edges of vertical blocks.

To prove that this construction is a bijection, we describe its inverse. 
Any triple $(u,m,d)$ such as in Definition \ref{dfn:NILP} corresponds to a unique Baxter slicing of a parallelogram polyomino $P$, 
whose contour is defined by $u$ and $d$ and whose block division is obtained by $m$. 
More precisely, the lower (resp. upper) border of $P$ is the path $E\cdot d \cdot N$ (resp. $N \cdot u \cdot E$) drawn starting at $(0,0)$. 
Let the starting point of the path $m$ be $(1,1)$. Then, the blocks inside $P$ are drawn according to the steps of $m$ as follows.
For every $E$ step $s$ in $m$, draw a vertical block whose top edge is $s$ 
and that extends downwards until the border of $P$. 
Similarly, for every $N$ step $s$ in $m$, draw a horizontal block whose right edge is $s$ 
and that extends leftwards until the border of $P$.
And finally, add the initial block consisting of one cell extending from $(0,0)$ to $(1,1)$.
\end{proof}

Up to the simple bijective correspondence described in Theorem~\ref{thm:bijection_Baxter_slicings_NILPs}, 
our Theorem~\ref{thm:growth_Baxter_slicings} can also be seen as a description of the growth of triples of NILPs 
according to the generating tree $\TBax$, which was already described in~\cite{mbm}. 
% vérifier que ces références sont nécessaires ailleurs. 
% To our knowledge, this was never described, although both triples of NILPs~ \cite{Viennot_SLC81} and the tree $\TBax$~\cite{BM} are well-known to correspond to Baxter numbers. 

\subsection{Baxter slicings of a given shape}

One of the most basic enumerative questions that one may ask about Baxter slicings is to determine 
the number of Baxter slicings whose shape is a given parallelogram polyomino $P$. 
In the light of the previous bijection between Baxter slicings and triples of NILPs, this question can be rephrased in terms of counting how many triples of NILPs correspond to given ``external" paths (\emph{i.e.} $u$ and $d$), which are the two paths defining $P$.
This is not the main focus of our work, so we just give the extremal cases as observations. 

\begin{observation}
Let $P$ be the parallelogram polyomino of rectangular shape, whose bounding rectangle has dimensions $k \times \ell$. 
The number of Baxter slicings of $P$ is $\binom{k+\ell-2}{\ell-1}$. 
\end{observation}

\begin{proof}
This follows from Theorem \ref{thm:bijection_Baxter_slicings_NILPs}, 
since the number of Baxter slicings of $P$ coincides with the number of paths from $(1,1)$ to $(k,\ell)$ using $N$ and $E$ steps. 
\end{proof}

\begin{observation}
Let $P$ be a snake, that is, a parallelogram polyomino not containing four cells placed as \tikz{
\begin{scope}[scale=.2]
\draw[help lines] (0,0) grid (2,2);
\end{scope}
}\,. 
There is only one Baxter slicing of $P$. 
\end{observation}

\begin{proof}
We prove that if $P$ is a snake of size $n$, then its interior is unambiguously divided in $n$ blocks, each consisting of a single cell. 
Since $P$ does not contain \tikz{
\begin{scope}[scale=.2]
\draw[help lines] (0,0) grid (2,2);
\end{scope}
}\,, then the topmost cell in the rightmost column is the only cell in its row or the only cell in its column. 
In the former (resp.~latter) case, it forms a horizontal (resp.~vertical) block. 
Removing this block from $P$, the remaining cells form a snake of size $n-1$, 
and the result follows by induction. 
\end{proof}

\section{Schr\"oder slicings}
\label{sec:Schroeder}

Our first interest in defining Baxter slicings is to find a family of objects enumerated by the Schr\"oder numbers 
which lie between parallelogram polyominoes and Baxter slicings, 
and which grow according to a succession rule that generalizes~\eqref{rule:Catalan} while specializing~\eqref{rule:BaxterMBM}. 
Note that to our knowledge, out of the many succession rules for Schr\"oder numbers~\cite{PS,West},
none has this property.

\subsection{Specializations and generalizations of succession rules\label{sec:gen-spec}}

Earlier in this article, we have compared the succession rules for Catalan numbers and Baxter numbers, 
and described how the latter generalizes the former. 
To deal with this first example, it was enough to stay at an informal level of what we mean by ``generalize''. 
However, in the remainder of the paper, we wish to compare more succession rules, 
and we offer to that effect a formalization of this informal idea that a succession rule generalizes (or conversely specializes) another one.

The following definition of generalizing (resp. specializing) a succession rule is more intended as a suggestion the reader may reflect on, than as a proper and accurate definition. 
This proposed definition is rather restrictive, and we shall see later some examples of situations that it does not encapsulate. 
These examples indeed do not fit into our idea of what generalization/specialization of succession rules should be. 
Despite the restrictive character of our proposed definition, we believe that it applies to all our examples of the current paper, and of the other papers~\cite{BijectionNick,SemiAndStrong}. 
We leave open the questions whether the ``correct'' definition should be a bit less restrictive to allow for more instances to fit in, 
and whether it should be on the contrary more restrictive, to prevent other undesirable examples.

\medskip

Consider two succession rules $\Omega_{\mathcal{A}}$ and $\Omega_{\mathcal{B}}$. 
These rules are to be thought of as encoding growths of combinatorial classes $\mathcal{A}$ and $\mathcal{B} \subseteq \mathcal{A}$, respectively. 
We however do not refer at all to the classes $\mathcal{A}$ and $\mathcal{B}$ in defining that $\Omega_{\mathcal{A}}$ generalizes $\Omega_{\mathcal{B}}$, 
but rather work directly on these succession rules.
%,in a way that is reminiscent of what we do in the proof of Theorem~\ref{thm:SchInBax} later in this paper. 
To say that $\Omega_{\mathcal{A}}$ {\em generalizes} $\Omega_{\mathcal{B}}$ (or equivalently, that $\Omega_{\mathcal{B}}$ {\em specializes} $\Omega_{\mathcal{A}}$), 
we require: \\ 
\hspace*{0.4cm}(1) the existence of a comparison relation ``smaller than or equal to'' (to be defined specifically on each example) 
between the labels\footnote{Labels are integers or ordered pairs of integers in the current paper, 
but they may be more complicated structures in general.} of $\Omega_{\mathcal{B}}$ and those of $\Omega_{\mathcal{A}}$, 
and, \\
\hspace*{0.4cm}(2) for any labels $\ell_A$ of $\Omega_{\mathcal{A}}$ and $\ell_B$ of $\Omega_{\mathcal{B}}$ with $\ell_B$ smaller than or equal to $\ell_A$, 
a way of mapping the productions of the label $\ell_B$ in $\Omega_{\mathcal{B}}$ to a subset of the productions of the label $\ell_A$ in $\Omega_{\mathcal{A}}$, 
such that a label is always mapped to a larger or equal one. 

The emblematic example of this definition is given by rule~\eqref{rule:Catalan} as specialization of rule~\eqref{rule:BaxterMBM}. 
A label $\ell_B$ of rule~\eqref{rule:Catalan} is an array of one integer value, say $(j)$, 
whereas a label $\ell_A$ of rule~\eqref{rule:BaxterMBM} is an array of two integer values, say $(h,k)$. 
The relation ``smaller than or equal to'' between $\ell_B$ and $\ell_A$ is defined by $j = k$. 
(Note that one could symmetrically  consider the comparison between $j$ and the first component $h$ of $\ell_A$.)
Moreover, the way of mapping the productions of the label $\ell_B$ in~\eqref{rule:Catalan} to a subset of the productions of the label $\ell_A$ in~\eqref{rule:BaxterMBM}
so that any label is always mapped to a larger or equal one is defined as follows:
\[
\begin{array}{lcccccccc}
\textrm{root labeled }(1)&\textrm{and }&(k)&\underset{\eqref{rule:Catalan}}{\rightsquigarrow} &(1),&(2),& \ldots &(k),&(k+1).\\
\qquad\quad\big\downarrow &&\big\downarrow & & \big\downarrow &\big\downarrow &\ldots&\big\downarrow & \big\downarrow\\[.5em]
\textrm{root labeled }(1,1)&\textrm{and }&(h,k)& \underset{\eqref{rule:BaxterMBM}}{\rightsquigarrow}& (h+1,1),& (h+1,2),& \ldots& (h+1,k),& (1,k+1).
\end{array}
\]
Note that only the relevant subset of the productions of $(h,k)$ for rule~\eqref{rule:BaxterMBM} has been displayed in the second row of the above table. 

\medskip

The above definition allows to identify a canonical embedding of the generating tree associated with $\Omega_{\mathcal{B}}$ 
into the  generating tree associated with  $\Omega_{\mathcal{A}}$.
More precisely, let us denote by $\mathcal{T}_{\mathcal{A}}$ and $\mathcal{T}_{\mathcal{B}}$ the two generating trees 
associated with the classes $\mathcal{A}$ and $\mathcal{B}$ and their growths according to $\Omega_{\mathcal{A}}$ and $\Omega_{\mathcal{B}}$. 
Then, induction shows that the above definition implies the existence of an injection $\phi$ 
(corresponding to the mapping of labels in item (2) of the definition) 
from the set of vertices of $\mathcal{T}_{\mathcal{B}}$ to the set of vertices of $\mathcal{T}_{\mathcal{A}}$ 
which preserves the level and the parent-child relation, 
such that for any vertex $v$ of $\mathcal{T}_{\mathcal{B}}$, the label of $v$ is smaller than or equal to the label of $\phi(v)$ in $\mathcal{T}_{\mathcal{A}}$. 
This injection allows to define the {\em ``canonical'' subtree} of $\mathcal{T}_{\mathcal{A}}$ isomorphic to $\mathcal{T}_{\mathcal{B}}$. 

\medskip

We conclude this section about a proposed definition of generalization/specialization of succession rules with 
examples that do not fit into our proposed definition. 
If readers wish to consider variants of the above definition, it is important to keep these examples in mind: 
indeed, they display situations which we do not want to enter our framework. 

First, to say that $\Omega_{\mathcal{A}}$ generalizes $\Omega_{\mathcal{B}}$, 
it is not enough to know generating trees for combinatorial classes $\mathcal{A}$ and $\mathcal{B} \subseteq \mathcal{A}$, 
encoded by succession rules $\Omega_{\mathcal{A}}$ and $\Omega_{\mathcal{B}}$ respectively. 
Indeed, it may be the case that the underlying growths for the classes $\mathcal{A}$ and $\mathcal{B}$ have nothing in common. 
This applies for instance to Dyck and Motzkin paths, with their growths presented in~\cite{Eco}, 
or to families of pattern-avoiding inversion sequences (namely, avoiding the triple of relations $(\geq,-,\geq)$ and $(\geq,\geq,>)$, respectively) 
with their growths defined in~\cite{BijectionNick}.  

Second, we do not want the definition of $\Omega_{\mathcal{A}}$ generalizing $\Omega_{\mathcal{B}}$ 
to be dependent of the combinatorial classes $\mathcal{A}$ and $\mathcal{B}$. 
Namely, consider the case where the growth for $\mathcal{B}\subseteq \mathcal{A}$ corresponding to $\Omega_{\mathcal{B}}$ 
specializes the growth for $\mathcal{A}$ corresponding to $\Omega_{\mathcal{A}}$, 
in the sense that for any object $b$ of $\mathcal{B}$ (which is of course also an object of $\mathcal{A}$),
the set of {\em active sites} of $b$ as an object of  $\mathcal{B}$ 
is a subset of the set of active sites of $b$ as an object of $\mathcal{A}$. 
This does not guarantee that $\Omega_{\mathcal{A}}$ generalizes $\Omega_{\mathcal{B}}$ 
for our proposed definition. 
Indeed, it may be the case that the active sites are not encoded in the same way in the labels of $\Omega_{\mathcal{A}}$ and $\Omega_{\mathcal{B}}$. 
This happens for instance for the separable permutations (growth described in~\cite{West})
and the Baxter permutations (growth described in~\cite{BM}). 

\begin{landscape}
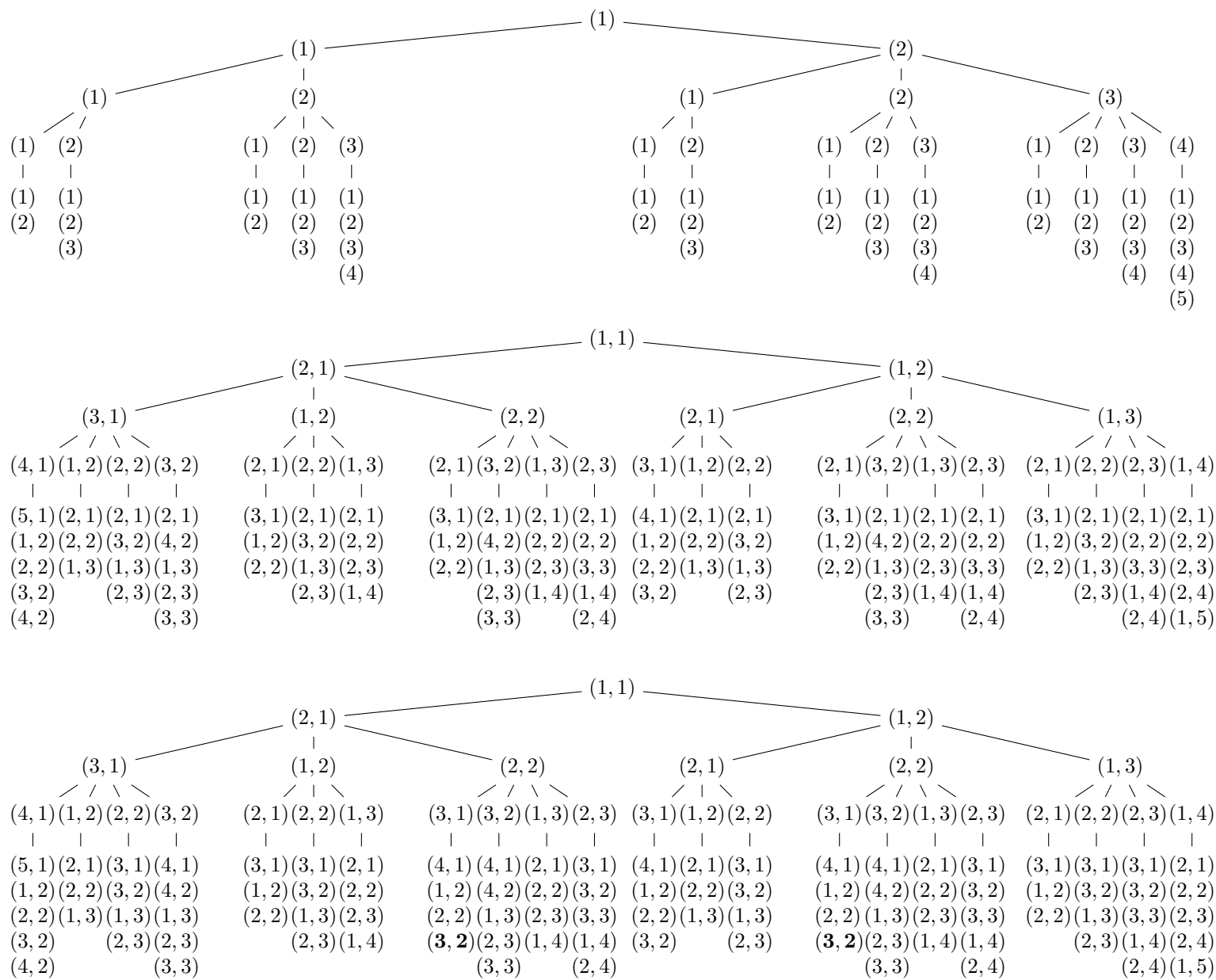
\begin{figure}[ht]
\vspace*{-1.5cm}
%%%%%%%Catalan
\begin{tikzpicture}[level distance=1.2cm]
\tikzstyle{level 1}=[sibling distance=10cm,level distance=0.5cm]
\tikzstyle{level 2}=[sibling distance=35mm,level distance=0.8cm]
\tikzstyle{level 3}=[sibling distance=8mm]
\tikzstyle{level 4}=[sibling distance=8mm,level distance=1.7cm]
\node {$(1)$} 
  child{ node {$(1)$}
    child { node {$(1)$} 
    child { node {$(1)$}child {node
    {$\begin{array}{l}
    (1)\\
    (2)\\~\\~\\~\end{array}$}}}
     child { node {$(2)$} child {node
    {$\begin{array}{l}
    (1)\\
    (2)\\(3)\\~\\~\end{array}$}}}
    child [missing]
    child [missing]
}
    child { node {$(2)$}
    child { node {$(1)$}child {node
    {$\begin{array}{l}
    (1)\\
    (2)\\~\\~\\~\end{array}$}}}
    child { node {$(2)$} child {node
    {$\begin{array}{l}
    (1)\\
    (2)\\(3)\\~\\~\end{array}$}}}
    child { node {$(3)$} child {node
    {$\begin{array}{l}
    (1)\\
    (2)\\(3)\\(4)\\~\end{array}$}}}}
    child [missing]
  }
  child{ node {$(2)$}
    child { node {$(1)$} 
    child { node {$(1)$} child {node
    {$\begin{array}{l}
    (1)\\
    (2)\\~\\~\\~\end{array}$}}}
    child { node {$(2)$} child {node
    {$\begin{array}{l}
    (1)\\
    (2)\\(3)\\~\\~\end{array}$}}}
    child [missing]
    }
    child { node {$(2)$}
    child { node {$(1)$} child {node
    {$\begin{array}{l}
    (1)\\
    (2)\\~\\~\\~\end{array}$}}}
    child { node {$(2)$} child {node
    {$\begin{array}{l}
    (1)\\
    (2)\\(3)\\~\\~\end{array}$}}}
    child { node {$(3)$} child {node
    {$\begin{array}{l}
    (1)\\
    (2)\\(3)\\(4)\\~\end{array}$}}}
    child [missing]
  }
    child { node {$(3)$} 
    child { node {$(1)$} child {node
    {$\begin{array}{l}
    (1)\\
    (2)\\~\\~\\~\end{array}$}}}
    child { node {$(2)$} child {node
    {$\begin{array}{l}
    (1)\\
    (2)\\(3)\\~\\~\end{array}$}}}
    child { node {$(3)$}child {node
    {$\begin{array}{l}
    (1)\\
    (2)\\(3)\\(4)\\~\end{array}$}}}
    child {node {$(4)$}child {node
    {$\begin{array}{l}
    (1)\\
    (2)\\(3)\\(4)\\(5)\end{array}$}}}
}
  }
;
\end{tikzpicture}
%%%%%%Schroder
\begin{tikzpicture}[level distance=1.2cm]
\tikzstyle{level 1}=[sibling distance=10cm,level distance=0.5cm]
\tikzstyle{level 2}=[sibling distance=35mm,level distance=0.8cm]
\tikzstyle{level 3}=[sibling distance=8mm]
\tikzstyle{level 4}=[sibling distance=8mm,level distance=1.7cm]
\node {$(1,1)$}
      child{ node {$(2,1)$}
      child { node {$(3,1)$}
      child  { node {$(4,1)$}child {node
    {$\begin{array}{l}
    (5,1)\\
    (1,2)\\
    (2,2)\\(3,2)\\(4,2)\end{array}$}}}
    child { node {$(1,2)$} child {node
    {$\begin{array}{l}
    (2,1)\\(2,2)\\(1,3)\\~\\~\end{array}$}}}
    child { node {$(2,2)$} child {node
    {$\begin{array}{l}(2,1)\\(3,2)\\
    (1,3)\\
    (2,3)\\~\end{array}$}}}
    child  { node {$(3,2)$}child {node
    {$\begin{array}{l}(2,1)\\(4,2)\\
    (1,3)\\
    (2,3)\\(3,3) \end{array}$}}} 
    }
    child { node {$(1,2)$}
    child { node {$(2,1)$} child {node
    {$\begin{array}{l}(3,1)\\ 
    (1,2)\\
    (2,2)\\~\\~\end{array}$}}}
    child  { node {$(2,2)$}child {node
    {$\begin{array}{l}(2,1)\\(3,2)\\
    (1,3)\\
    (2,3)\\~\end{array}$}}}
    child { node {$(1,3)$} child {node
    {$\begin{array}{l}
    (2,1)\\(2,2)\\ (2,3)\\(1,4)\\~\end{array}$}}}}
    child { node {$(2,2)$}  
    child  { node {$(2,1)$} child {node
    {$\begin{array}{l}(3,1)\\
    (1,2)\\
    (2,2)\\ ~\\~\end{array}$}}} 
     child  { node {$(3,2)$}child {node
    {$\begin{array}{l}(2,1)\\(4,2)\\
    (1,3)\\
    (2,3)\\(3,3)\end{array}$}}}
    child { node {$(1,3)$} child {node
    {$\begin{array}{l}
    (2,1)\\(2,2)\\ (2,3)\\ (1,4)\\~\end{array}$}}}
    child { node {$(2,3)$} child {node
    {$\begin{array}{l}(2,1)\\(2,2)\\(3,3)\\
    (1,4)\\
    (2,4)\end{array}$}}}
    }
  }
    child{ node {$(1,2)$}
    child { node {$(2,1)$} 
    child  { node {$(3,1)$}
     child {node
    {$\begin{array}{l}(4,1)\\
    (1,2)\\
    (2,2)\\(3,2)\\~\end{array}$}}} 
    child { node {$(1,2)$} 
    child {node
    {$\begin{array}{l}
    (2,1)\\(2,2)\\(1,3)\\~\\~\end{array}$}}}
    child  { node {$(2,2)$}child {node
    {$\begin{array}{l}(2,1)\\(3,2)\\
    (1,3)\\
    (2,3)\\~\end{array}$}}}
     }
    child  { node {$(2,2)$}
    child  { node {$(2,1)$}child {node
    {$\begin{array}{l}(3,1)\\ 
    (1,2)\\
    (2,2)\\~\\~\end{array}$}}}
     child  { node {$(3,2)$} child {node
    {$\begin{array}{l}(2,1)\\(4,2)\\
    (1,3)\\
    (2,3)\\(3,3) \end{array}$}}}
    child { node {$(1,3)$} child {node
    {$\begin{array}{l}   
    (2,1)\\(2,2)\\(2,3)\\(1,4)\\~\end{array}$}}}
    child { node {$(2,3)$} child {node
    {$\begin{array}{l}(2,1)\\(2,2)\\(3,3)\\
    (1,4)\\
    (2,4)\\\end{array}$}}}
    }
    child { node {$(1,3)$} 
    child { node {$(2,1)$} 
    child {node
    {$\begin{array}{l}(3,1)\\
    (1,2)\\
    (2,2)\\ ~\\~\end{array}$}}}
    child  { node {$(2,2)$}
    child {node
    {$\begin{array}{l}(2,1)\\(3,2)\\
    (1,3)\\
    (2,3)\\~\end{array}$}}}
child  { node {$(2,3)$}
child {node
    {$\begin{array}{l}(2,1)\\(2,2)\\(3,3)\\
    (1,4)\\
    (2,4)\end{array}$}}}   
    child { node {$(1,4)$} 
    child {node
    {$\begin{array}{l}
    (2,1)\\(2,2)\\(2,3)\\(2,4)\\(1,5)\end{array}$}}}}}
;
\end{tikzpicture}

\vspace{0.5cm}

%%%%%%%Baxter
\begin{tikzpicture}[level distance=1.2cm]
\tikzstyle{level 1}=[sibling distance=10cm,level distance=0.5cm]
\tikzstyle{level 2}=[sibling distance=35mm,level distance=0.8cm]
\tikzstyle{level 3}=[sibling distance=8mm]
\tikzstyle{level 4}=[sibling distance=8mm,level distance=1.7cm]
\node {$(1,1)$}
      child{ node {$(2,1)$}
      child { node {$(3,1)$}
      child  { node {$(4,1)$}child {node
    {$\begin{array}{l}
    (5,1)\\
    (1,2)\\
    (2,2)\\(3,2)\\(4,2)\end{array}$}}}
    child { node {$(1,2)$} child {node
    {$\begin{array}{l}
    (2,1)\\(2,2)\\(1,3)\\~\\~\end{array}$}}}
    child { node {$(2,2)$} child {node
    {$\begin{array}{l}(3,1)\\(3,2)\\
    (1,3)\\
    (2,3)\\~\end{array}$}}}
    child  { node {$(3,2)$}child {node
    {$\begin{array}{l}(4,1)\\(4,2)\\
    (1,3)\\
    (2,3)\\(3,3) \end{array}$}}} 
    }
    child { node {$(1,2)$}
    child { node {$(2,1)$} child {node
    {$\begin{array}{l}(3,1)\\ 
    (1,2)\\
    (2,2)\\~\\~\end{array}$}}}
    child  { node {$(2,2)$}child {node
    {$\begin{array}{l}(3,1)\\(3,2)\\
    (1,3)\\
    (2,3)\\~\end{array}$}}}
    child { node {$(1,3)$} child {node
    {$\begin{array}{l}
    (2,1)\\(2,2)\\ (2,3)\\(1,4)\\~\end{array}$}}}}
    child { node {$(2,2)$}  
    child  { node {$(3,1)$} child {node
    {$\begin{array}{l}(4,1)\\
    (1,2)\\
    (2,2)\\ \mathbf{(3,2)}\\~\end{array}$}}} 
     child  { node {$(3,2)$}child {node
    {$\begin{array}{l}(4,1)\\(4,2)\\
    (1,3)\\
    (2,3)\\(3,3)\end{array}$}}}
    child { node {$(1,3)$} child {node
    {$\begin{array}{l}
    (2,1)\\(2,2)\\ (2,3)\\ (1,4)\\~\end{array}$}}}
    child { node {$(2,3)$} child {node
    {$\begin{array}{l}(3,1)\\(3,2)\\(3,3)\\
    (1,4)\\
    (2,4)\end{array}$}}}
    }}
    child{ node {$(1,2)$}
    child { node {$(2,1)$} 
    child  { node {$(3,1)$}
     child {node
    {$\begin{array}{l}(4,1)\\
    (1,2)\\
    (2,2)\\(3,2)\\~\end{array}$}}} 
    child { node {$(1,2)$} 
    child {node
    {$\begin{array}{l}
    (2,1)\\(2,2)\\(1,3)\\~\\~\end{array}$}}}
    child  { node {$(2,2)$}child {node
    {$\begin{array}{l}(3,1)\\(3,2)\\
    (1,3)\\
    (2,3)\\~\end{array}$}}}
     }
    child  { node {$(2,2)$}
    child  { node {$(3,1)$}child {node
    {$\begin{array}{l}(4,1)\\ 
    (1,2)\\
    (2,2)\\\mathbf{(3,2)}\\~\end{array}$}}}
     child  { node {$(3,2)$} child {node
    {$\begin{array}{l}(4,1)\\(4,2)\\
    (1,3)\\
    (2,3)\\(3,3) \end{array}$}}}
    child { node {$(1,3)$} child {node
    {$\begin{array}{l}   
    (2,1)\\(2,2)\\(2,3)\\(1,4)\\~\end{array}$}}}
    child { node {$(2,3)$} child {node
    {$\begin{array}{l}(3,1)\\(3,2)\\(3,3)\\
    (1,4)\\
    (2,4)\\\end{array}$}}}
    }
    child { node {$(1,3)$} 
    child { node {$(2,1)$} 
    child {node
    {$\begin{array}{l}(3,1)\\
    (1,2)\\
    (2,2)\\ ~\\~\end{array}$}}}
    child  { node {$(2,2)$}
    child {node
    {$\begin{array}{l}(3,1)\\(3,2)\\
    (1,3)\\
    (2,3)\\~\end{array}$}}}
     child  { node {$(2,3)$}
     child {node
    {$\begin{array}{l}(3,1)\\(3,2)\\(3,3)\\
    (1,4)\\
    (2,4)\end{array}$}}}   
    child { node {$(1,4)$} 
    child {node
    {$\begin{array}{l}
    (2,1)\\(2,2)\\(2,3)\\(2,4)\\(1,5)\end{array}$}}}}}
;
\end{tikzpicture}
\caption{The first levels of the generating trees for rules~\eqref{rule:Catalan},~\eqref{rule:newSchroeder} and~\eqref{rule:BaxterMBM}. 
Bold characters are used to indicate the first vertices of $\TBax$ that do not appear in $\TSch$. \label{fig:GeneratingTrees}}
\end{figure}
\end{landscape}

\subsection{A new Schr\"oder succession rule}

Let us consider the following succession rule, whose associated generating tree is denoted $\TSch$ (shown in Figure~\ref{fig:GeneratingTrees} page~\pageref{fig:GeneratingTrees}): 
\begin{equation}
\textrm{root labeled } (1,1) \textrm{ \quad and \quad } (h,k) \rightsquigarrow \begin{cases}
(1, k + 1), (2, k + 1), \ldots , (h, k + 1), \\
(2,1), (2,2), \ldots, (2,k-1),\\
 (h + 1, k).
\end{cases} \tag{\textrm{NewSch}} \label{rule:newSchroeder}
\end{equation}

\begin{theorem}
The enumeration sequence associated with rule~\eqref{rule:newSchroeder} is that of Schr\"oder numbers. 
\end{theorem}

\begin{proof}
From~\cite{West}, 
we know that the following succession rule is associated with Schr\"oder numbers: 
\begin{equation}
\textrm{root labeled } (2) \textrm{ \quad and \quad } (j) \rightsquigarrow (3), (4), \ldots , (j), (j + 1), (j + 1). \tag{\textrm{Sch}} \label{rule:SchroederWest}
\end{equation}
We claim that rules~\eqref{rule:newSchroeder} and~\eqref{rule:SchroederWest} produce the same generating tree. 
Indeed, replacing each label $(h,k)$ in rule~\eqref{rule:newSchroeder} by the sum $h+k$ of its elements immediately gives rule~\eqref{rule:SchroederWest}. 
\end{proof}

What is further interesting with rule~\eqref{rule:newSchroeder} is that rule~\eqref{rule:BaxterMBM} for Baxter numbers generalizes it and rule~\eqref{rule:Catalan} for Catalan numbers specializes it. Indeed, it is not obvious that rule~\eqref{rule:SchroederWest} generalizes rule~\eqref{rule:Catalan}, ensuring that $\TSch$ contains a ``canonical'' subtree isomorphic to $\TCat$ -- see generating trees depicted in Figure~\ref{fig:GeneratingTrees}. Yet this fact becomes clear with rule~\eqref{rule:newSchroeder}, which can be immediately seen to generalize rule~\eqref{rule:Catalan}, in the same fashion rule~\eqref{rule:BaxterMBM} does.

\begin{theorem}
\label{thm:SchInBax}
The succession rule~\eqref{rule:newSchroeder} generalizes rule~\eqref{rule:Catalan}, while specializing rule~\eqref{rule:BaxterMBM}. Hence, $\TCat$ is (isomorphic to) a subtree of $\TSch$, which in turn is (isomorphic to) a subtree of $\TBax$.
\end{theorem}

%Note that it is not obvious that rule~\eqref{rule:SchroederWest} generalizes rule~\eqref{rule:Catalan}, ensuring that $\TSch$ contains a subtree isomorphic to $\TCat$. 
%But this becomes clear with rule~\eqref{rule:newSchroeder}, which can be immediately seen to generalize rule~\eqref{rule:Catalan},in the same fashion rule~\eqref{rule:BaxterMBM} does. 
%Indeed, in rule~\eqref{rule:newSchroeder}, looking only at the productions $(2,1),(2,2),\ldots, (2,k-1),(h+1,k)$ and $(1,k+1)$ of a label $(h,k)$, and considering the second component of the labels, we recover rule~\eqref{rule:Catalan}. 

\begin{proof}
Note first that the only difference between rules~\eqref{rule:BaxterMBM} and~\eqref{rule:newSchroeder} 
is that labels $(h+1,i)$ for $1 \leq i \leq k-1$ in the production of rule~\eqref{rule:BaxterMBM} 
are replaced by $(2,i)$ in rule~\eqref{rule:newSchroeder}. 
Thus, first we prove that rule~\eqref{rule:newSchroeder} generalizes rule~\eqref{rule:Catalan} in the same way as rule~\eqref{rule:BaxterMBM} does. 
From this fact, it follows that $\TCat$ is isomorphic to a subtree of $\TSch$ 
(the one called canonical subtree in Section~\ref{sec:gen-spec}).

More precisely, we define that a label $(j)$ for rule~\eqref{rule:Catalan} is smaller than or equal to a label $(h,k)$ in rule~\eqref{rule:newSchroeder}
when $j=k$. Then, we consider the subset $(2,1),(2,2),\ldots, (2,k-1),(h+1,k)$ and $(1,k+1)$ of the productions of a label $(h,k)$ by rule~\eqref{rule:newSchroeder}, 
whose second components give the productions of $(k)$ for rule~\eqref{rule:Catalan}. 
Consequently, the mapping below witnesses the fact that~\eqref{rule:Catalan} specializes~\eqref{rule:newSchroeder}:
\[
\begin{array}{lcccccccc}
\textrm{root labeled }(1)&\textrm{and }&(k)&\underset{\eqref{rule:Catalan}}{\rightsquigarrow} &(1),& \ldots&(k-1), &(k),&(k+1).\\
\qquad\quad\big\downarrow &&\big\downarrow & & \big\downarrow &\ldots&\big\downarrow&\big\downarrow & \big\downarrow\\[.5em]
\textrm{root labeled }(1,1)&\textrm{and }&(h,k)& \underset{\eqref{rule:newSchroeder}}{\rightsquigarrow}& (2,1),& \ldots& (2,k-1),& (h+1,k),& (1,k+1).
\end{array}
\]

\medskip

Now, we prove that rule~\eqref{rule:newSchroeder} specializes rule~\eqref{rule:BaxterMBM}, 
from which, as before, it follows that $\TSch$ is isomorphic to a subtree of $\TBax$. 
Observe that labels in rule~\eqref{rule:newSchroeder} are arrays of two integer values as well as those in rule~\eqref{rule:BaxterMBM}. 
We define that a label $(h,k)$ for rule~\eqref{rule:newSchroeder} is smaller than or equal to a label $(h',k')$ for rule~\eqref{rule:BaxterMBM} 
when $h \leq h'$ and $k=k'$. 
So, to conclude that~\eqref{rule:BaxterMBM} generalizes~\eqref{rule:newSchroeder}, 
we just need to exhibit a mapping which respects this order of the productions of a label $(h,k)$ in~\eqref{rule:newSchroeder}
to a subset of the productions of the label $(h',k)$ in~\eqref{rule:BaxterMBM}, 
for any $h' \geq h$. 
This mapping is given by 
\[
\textrm{root labeled }(1,1)\quad \longrightarrow \quad \textrm{root labeled }(1,1), \textrm{ and }
\]
\begin{multline*}
(h,k) \underset{\textrm{NewSch}}{\rightsquigarrow} (1,k+1), \ldots ,(h,k+1),  \hspace{3.5cm} (2,1),  \hspace{0.5cm} \ldots \hspace{0.5cm}(2,k-1),\; (h+1,k). \\
\hspace{2cm} \big\downarrow \hspace{0.6cm} \ldots \hspace{0.5cm} \big\downarrow \hspace{4.8cm} \big\downarrow \hspace{0.8cm} \ldots \hspace{1cm} \big\downarrow \hspace{1.3cm} \big\downarrow\\
\hspace{0.2cm}(h',k)  \hspace{0.15cm} \underset{\textrm{Bax}}{\rightsquigarrow}  \hspace{0.15cm} (1,k+1), \ldots (h,k+1), (h+1,k+1), \ldots (h',k+1),\; (h'+1,1), \ldots (h'+1,k-1),\; (h'+1,k).
\end{multline*}

%With this remark, we can prove the following statement by induction on the depth of the vertices in the generating trees: 
%for any $h,k$, and $h'\geq h$, there exists an injective mapping from the vertices of the generating tree produced from root $(h,k)$ in rule~\eqref{rule:newSchroeder}
%to the vertices of a subtree of the generating tree produced from root $(h',k)$ in rule~\eqref{rule:BaxterMBM}, 
%which preserves the depth, and such that for any vertex labeled $(i,j)$, its image is labeled $(i',j)$ for some $i' \geq i$. 
%Indeed, it is enough to map vertices of the generating trees along the productions of rules~\eqref{rule:BaxterMBM} and~\eqref{rule:newSchroeder} as follows: 
%\begin{multline*}
%(h,k) \underset{\textrm{NewSch}}{\rightsquigarrow} (1,k+1), \ldots ,(h,k+1),  \hspace{3.5cm} (2,1),  \hspace{0.5cm} \ldots \hspace{0.5cm}(2,k-1),\; (h+1,k). \\
%\hspace{2cm} \big\downarrow \hspace{0.6cm} \ldots \hspace{0.5cm} \big\downarrow \hspace{4.8cm} \big\downarrow \hspace{0.8cm} \ldots \hspace{1cm} \big\downarrow \hspace{1.3cm} \big\downarrow\\
%\hspace{0.2cm}(h',k)  \hspace{0.15cm} \underset{\textrm{Bax}}{\rightsquigarrow}  \hspace{0.15cm} (1,k+1), \ldots (h,k+1), (h+1,k+1), \ldots (h',k+1),\; (h'+1,1), \ldots (h'+1,k-1),\; (h'+1,k).
%\end{multline*}
\end{proof}

To our knowledge, this is the first time three succession rules for Catalan, Schr\"oder and Baxter numbers are given, 
which are each a generalization of the previous one. 
The first levels of the generating trees for rules~\eqref{rule:Catalan},~\eqref{rule:newSchroeder} and~\eqref{rule:BaxterMBM} 
are shown in Figure~\ref{fig:GeneratingTrees} on page~\pageref{fig:GeneratingTrees}. 

\subsection{Definition of Schr\"oder slicings, and their growth}

We want to define Schr\"oder slicings so that they form a subset of the Baxter slicings that is enumerated by the Schr\"oder numbers, 
and whose growth is described by rule~\eqref{rule:newSchroeder}. 
To do that, recall that a ``canonical'' subtree of $\TBax$ isomorphic to $\TSch$ was built in the proof of Theorem~\ref{thm:SchInBax}. 
From there, it is enough to label the vertices of $\TBax$ by the corresponding Baxter slicings, 
and to keep only the objects which label a vertex of this ``canonical'' subtree. 
With this global approach to the definition of Schr\"oder slicings, the problem is  
to provide a characterization of these objects that would be local, \emph{i.e.}~that could be checked on any given Baxter slicing 
without reconstructing the whole chain of productions according to rule~\eqref{rule:BaxterMBM} that resulted in this object. 

For the sake of clarity, we have chosen to reverse the order in the presentation of Schr\"oder slicings:
we will first give their ``local characterization'', 
and then prove that they grow according to rule~\eqref{rule:newSchroeder}. 
It will be clear in the proof of this statement (see Theorem~\ref{thm:SchroederSlicings_growth}) 
that Schr\"oder slicings correspond to the ``canonical'' subtree of $\TBax$ on Baxter slicings described earlier. 

\begin{definition}
Let $B$ be a Baxter slicing of a parallelogram polyomino $P$, and let $u$ be a horizontal block of $B$. 
We denote by $\ell(u)$ the width of $u$. 
The projection $X(u)$ of $u$ on the lower border of $P$ is the lower-most point of this border 
whose abscissa is that of the right edge of $u$. 
We now define $r(u)$ to be the number of horizontal steps on the lower border of $P$ to the left of $X(u)$ 
before a vertical step (or the bottom-left corner of $P$) is met.
\label{dfn:ell_and_r}
\end{definition}

\begin{figure}[ht]
\begin{center}
\includegraphics[scale=0.8]{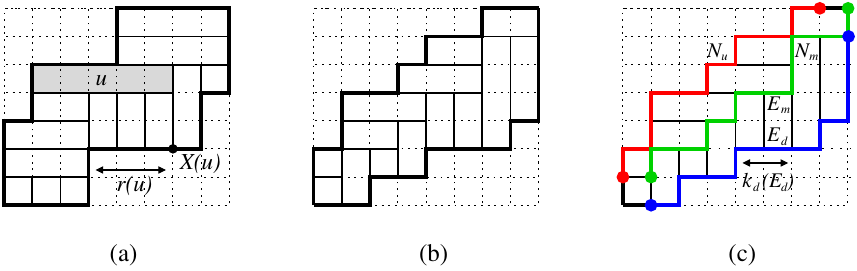}
\end{center}
\caption{(a) Illustration of Definition~\ref{dfn:ell_and_r}, (b) example of Schr\"oder slicing, 
and (c) illustration of Definition~\ref{dfn:pair_steps_matched} and Theorem~\ref{thm:bijection_to_Schroeder_NILPs}. \label{fig:SchroederSlicings}}
\end{figure}

\begin{definition}
A Schr\"oder slicing is any Baxter slicing such that for any horizontal block $u$, the following inequality holds: 
\begin{equation}
\ell(u) \leq r(u) +1. \tag{$\mathnormal{\ell r_1}$} \label{eq:lr1}
\end{equation}
\label{dfn:SchroederSlicing}
\end{definition}

Figure~\ref{fig:SchroederSlicings}(a,b) illustrates the definitions of $\ell(u)$ and $r(u)$, 
and shows an example of a Schr\"oder slicing. 

\begin{theorem}
A generating tree for Schr\"oder slicings is $\TSch$, associated with rule~\eqref{rule:newSchroeder}. 
\label{thm:SchroederSlicings_growth}
\end{theorem}

\begin{proof}
Like Baxter slicings, Schr\"oder slicings grow by adding vertical blocks on the right and horizontal blocks on top, 
but whose width is restricted, so that condition~\eqref{eq:lr1} is always satisfied. 

To any Schr\"oder slicing $P$, let us associate the label $(h,k)$ 
where $h$ (resp.~$k$) denotes the maximal width (resp.~height) of a horizontal (resp.~vertical) block that may be added to $P$, 
without violating condition~\eqref{eq:lr1}. 
Note that if a horizontal block of width $i$ may be added, 
then for all $i' \leq i$, the addition of a horizontal block of width $i'$ is also allowed. 
Consequently, we may add horizontal blocks of width $1$ to $h$ to $P$.  
Notice also that $k$ denotes the height of the rightmost column of $P$ 
(since condition~\eqref{eq:lr1} introduces no restriction on vertical blocks), 
and that columns of any height from $1$ to $k$ may be added to $P$. 

Figure~\ref{fig:production_Schroeder_slicings} illustrates 
the three cases discussed below in the growth of Schr\"oder slicings according to rule~\eqref{rule:newSchroeder}. 

\begin{figure}[ht]
\begin{center}
\includegraphics[scale=0.8]{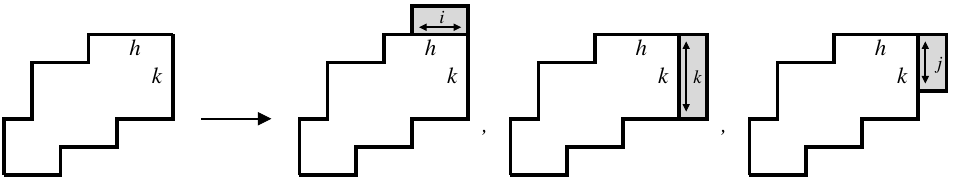}
\end{center}
\caption{The productions of a Schr\"oder slicing of label $(h,k)$ following rule~\eqref{rule:newSchroeder}. \label{fig:production_Schroeder_slicings}}
\end{figure}

For any $i \leq h$, consider the Schr\"oder slicing $P'$ obtained by adding a horizontal block $u$ of width $ \ell(u) = i$. 
We claim that the label of $P'$ is $(i,k+1)$. 
Obviously, the height of the last column of $P'$ is $k+1$. 
Moreover, if we were to add a further horizontal block $u'$ of any width $\ell(u') = i' \leq i$, 
$u'$ would satisfy condition~\eqref{eq:lr1}, since $X(u) = X(u')$ and $r(u) = r(u')$. 

Next, consider the Schr\"oder slicing $P'$ obtained by adding a column of height $k$ to $P$. 
We claim that it has label $(h+1,k)$. 
Of course, the rightmost column of $P'$ has height $k$. 
Moreover, the horizontal blocks $u'$ that may be added to $P'$ are of two types:  
either the block $u'$ is made of one single cell on top of the rightmost column of $P'$, 
or $u'$ is exactly the same as a horizontal block that could be added to $P$, 
except that it is augmented by one cell on the right. 
In this latter case, condition~\eqref{eq:lr1} is indeed satisfied since both $\ell(u')$ and $r(u')$ increase by $1$, when going from $P$ to $P'$. 

Finally, for any $j < k$, the Schr\"oder slicing $P'$ obtained by adding a column of height $j$ to $P$ has label $(2,j)$. 
Indeed, the rightmost column of $P'$ has height $j$, 
and only horizontal blocks $u'$ of width $1$ or $2$ may be added to $P'$ without violating condition~\eqref{eq:lr1}, 
since $r(u')=1$. 
\end{proof}

\section{Other Schr\"oder restrictions of Baxter objects}
\label{sec:other_Schroeder}

For any Baxter class $\mathcal{C}$ whose growth according to rule~\eqref{rule:BaxterMBM} is understood, 
it is immediate to define a Schr\"oder subclass of $\mathcal{C}$. 
Indeed, we can consider the full generating tree of shape $\TBax$ associated with $\mathcal{C}$, 
its ``canonical'' subtree isomorphic to $\TSch$, 
and keep only the objects of $\mathcal{C}$ associated with a vertex of $\TSch$. 
This method has the advantage of being systematic, 
but it does not \emph{a priori} provide a characterization of the objects in the Schr\"oder subclass 
which does not refer to the generating trees. 

In this section, we give three examples of Schr\"oder subclasses of Baxter classes. They have not been 
obtained with the above general method, but we provide for each of them a characterization of the
Schr\"oder objects without reference to generating trees.
% In this section, we give three examples of Schr\"oder subclasses of Baxter classes. 
% The first two are not obtained with the above general method: 
% we rather define the objects without reference to generating trees, 
% and then prove that they are counted by Schr\"oder numbers (note that the proof for the second example uses generating trees). 
% The third example is on the contrary obtained by considering the canonical subtree $\TSch$ of $\TBax$, 
% and we are nevertheless able to characterize the Schr\"oder objects independently of generating trees. 

\subsection{A Schr\"oder family of NILPs}

From Theorem~\ref{thm:bijection_Baxter_slicings_NILPs}, we have a simple bijection between triples of NILPs and Baxter slicings. 
And in Section~\ref{sec:Schroeder}, we have seen a subset of Baxter slicings enumerated by the Schr\"oder numbers. 
A natural question, which we now solve, is then to give a characterization of the triples of NILPs 
which correspond to Schr\"oder slicings via the bijection of Theorem~\ref{thm:bijection_Baxter_slicings_NILPs}. 

\begin{definition}
Let $(u,m,d)$ be a triple of NILPs as in Definition~\ref{dfn:NILP}. 

A pair $(N_u,N_m)$ of $N$ steps of $u$ and $m$ is \emph{matched} 
if there exists $i$ such that $N_u$ (resp.~$N_m$) is the $i$-th $N$ step of $u$ (resp.~$m$).
Similarly, a pair $(E_m,E_d)$ of $E$ steps of $m$ and $d$ is \emph{matched} 
if there exists $i$ such that $E_m$ (resp.~$E_d$) is the $i$-th $E$ step of $m$ (resp.~$d$).

Moreover, for any $N$ step $N_u$ in $u$, 
we denote by $h_u(N_u)$ the number of $E$ steps of $u$ that occur before $N_u$. 
Similarly, for any $N$ step $N_m$ in $m$, 
we denote by $h_m(N_m)$ the number of $E$ steps of $m$ that occur before $N_m$. 
And for any $E$ step $E_d$ in $d$, we denote by $k_d(E_d)$ the largest $k$ such that $E^k$ is a factor of $d$ ending in $E_d$. 
\label{dfn:pair_steps_matched}
\end{definition}

Figure~\ref{fig:SchroederSlicings}(c) (page~\pageref{fig:SchroederSlicings}) should help understand these definitions. 

\begin{definition}
A Schr\"oder triple of NILPs is any triple $(u,m,d)$ as in Definition~\ref{dfn:NILP} such that for any $N$ step $N_u$ of the path $u$, 
denoting $N_m$ the $N$ step of $m$ such that $(N_u,N_m)$ is matched, 
$E_m$ the last $E$ step of $m$ before $N_m$,
and $E_d$ the $E$ step of $d$ such that $(E_m,E_d)$ is matched, 
the following inequality holds:
\begin{equation}
h_m(N_m)-h_u(N_u) \leq k_d(E_d).\tag{$\star$}\label{eq:cond_NILPs}
\end{equation}
\label{dfn:SchroederNILPs}
\end{definition}

\begin{theorem}
Schr\"oder slicings are in one-to-one correspondence with Schr\"oder triples of NILPs 
by means of the size-preserving bijection described in Theorem~\ref{thm:bijection_Baxter_slicings_NILPs}.
\label{thm:bijection_to_Schroeder_NILPs}
\end{theorem}

\begin{proof}
We prove that the image of the class of Schr\"oder slicings under the bijection given in Theorem~\ref{thm:bijection_Baxter_slicings_NILPs} 
coincides with the class of Schr\"oder triples of NILPs of Definition~\ref{dfn:SchroederNILPs}. 
This will follow since condition~\eqref{eq:cond_NILPs} on triples of NILPs is equivalent to condition~\eqref{eq:lr1} on Baxter slicings.

Let $(u,m,d)$ be the image of a Baxter slicing $P$. 
By construction (see also Figure~\ref{fig:SchroederSlicings}(c)), every horizontal block $w$ of $P$ is associated with a pair $(N_u,N_m)$ of matched $N$ steps of $u$ and $m$, 
which correspond to the left (for $N_u$) and right (for $N_m$) edges of $w$. 
Similarly, every vertical block of $P$ is associated with a pair $(E_m,E_d)$ of matched $E$ steps of $m$ and $d$, 
corresponding to the upper and lower edges of the block. 

Consider a horizontal block $w$ in $P$, and let $(N_u,N_m)$ be the associated pair of matched steps. 
Denote by $E_m$ the last $E$ step of $m$ before $N_m$,
and by $E_d$ the $E$ step of $d$ such that $(E_m,E_d)$ is matched. 
This is the situation represented in Figure~\ref{fig:SchroederSlicings}(c).
We claim that $w$ satisfies condition~\eqref{eq:lr1} if and only if $N_u, N_m$ and $E_d$ satisfy condition~\eqref{eq:cond_NILPs}. 
On one hand, note that the width $\ell(w)$ of $w$ is also expressed as $h_m(N_m)+1-h_u(N_u)$. 
On the other hand, it is not hard to see that $r(w) = k_d(E_d)$. 
Indeed, the projection $X(w)$ of $w$ on the lower border of $P$ is the ending point of the step $E_d$ in $d$, 
so that both $r(w)$ and $k_d(E_d)$ denote 
the maximal number of $E$ (or horizontal) steps seen when reading $d$ (that is to say, the lower border of $P$) from right to left starting from $X(w)$. 
It follows that $\ell(w) \leq r(w) +1$ if and only if $h_m(N_m)-h_u(N_u) \leq k_d(E_d)$, which concludes the proof.
\end{proof}

\subsection{Another Schr\"oder subset of Baxter permutations}

Recall that a permutation $\sigma = \sigma_1 \sigma_2 \ldots \sigma_n$ contains a permutation $\tau = \tau_1 \tau_2 \ldots \tau_k$ (called \emph{pattern})
if there exists $i_1 < i_2 < \ldots < i_k$ such that $\sigma_{i_a} < \sigma_{i_b}$ if and only if $\tau_a < \tau_b$.
%Otherwise, $\sigma$ avoids $\tau$.
Moreover, recall that a permutation $\pi = \pi_1 \pi_2 \dots \pi_n$ contains the vincular pattern $\commonpat$ if there exists a subsequence $\pi_i \pi_j \pi_{j+1} \pi_k$ of $\pi$ (with $i<j<k-1$), 
called an \emph{occurrence} of the pattern, 
that satisfies $\pi_{j+1} < \pi_i < \pi_k < \pi_j$. 
Containment and occurrence of the pattern $\baxpat$ is defined similarly. 
A permutation not containing a pattern avoids it. 

Baxter permutations~\cite[among many others]{BM} are those that avoid both $\commonpat$ and $\baxpat$. 
The class of separable permutations, defined by the avoidance of $2413$ and $3142$, 
is a well-known subset of the set $\Bax$ of Baxter permutations and is enumerated by the Schr\"oder numbers. 
A generating tree for separable permutations following rule~\eqref{rule:SchroederWest} has been described in~\cite{West}, 
but we have not been able to explain the growth of separable permutations according to rule~\eqref{rule:newSchroeder}. 
However, restricting the growth of Baxter permutations according to rule~\eqref{rule:BaxterMBM}, 
we are able to describe a new subset of Baxter permutations, enumerated by the Schr\"oder numbers, 
and whose growth is governed by rule~\eqref{rule:newSchroeder}. 

As explained at the beginning of this section, 
a Schr\"oder subset of Baxter permutations can be obtained by considering the ``canonical'' embedding of $\TSch$ in $\TBax$. 
Doing so, the two Baxter permutations of size $5$ that are not obtained are $13254$ and $23154$, 
which correspond to the vertices of $\TBax$ shown in bold characters in Figure~\ref{fig:GeneratingTrees}. 
Although this subset of Baxter permutations is easy to define from the generating tree perspective, 
we have not been able to characterize the permutations it contains without referring to the generating trees, 
which is somewhat unsatisfactory. 
On the other hand, the subset of Baxter permutations studied below is not as immediate to define from the generating trees themselves, 
but has a nice characterization in terms of forbidden patterns. 
% {\color{blue}\textbf{[Veronica: Is it necessary to explain this? Note that a similar reasoning does not appear in floorplans.]}}
% Mathilde: I would keep it. It is probably good to explain it once, but does not have to be explained for all objects. 

The definition (in a special case) of bivincular patterns is useful to define the subset of Baxter permutations we are considering: 
a permutation $\sigma$ avoids the pattern $41323^+$ (resp.~$42313^+$) when no subsequence $\sigma_i \sigma_j \sigma_k \sigma_{\ell} \sigma_m$ of $\sigma$ satisfies 
$\sigma_j < \sigma_{\ell} < \sigma _k$ (resp.~$\sigma_{\ell} < \sigma_j < \sigma _k$), $\sigma_m = \sigma_k+1$, and $\sigma_m < \sigma_i$. 

\begin{theorem}
Let $\myS$ be the subset of Baxter permutations defined by avoidance of the (bi)vincular patterns $2\underbracket[.5pt][1pt]{41}3$, $3\underbracket[.5pt][1pt]{14}2$, $41323^+$ and $42313^+$. 
The generating tree obtained by letting permutations in $\myS$ grow by insertion of a maximal element is $\TSch$ 
(associated with rule~\eqref{rule:newSchroeder}), 
and consequently $\myS$ is enumerated by the Schr\"oder numbers. 
\label{thm:newSchroederPerms}
\end{theorem}

Note that the two Baxter permutations of size $5$ that are not in $\myS$ are $51324$ and $52314$. 
Figure~\ref{fig:growthSchroederSlicings} depicts the graphical representation of permutation $24351$, which belongs to $\myS$, and the set of permutations of $\myS$ obtained by adding a new maximum to it. 

\begin{figure}[ht]
\begin{center}
\includegraphics[scale=0.8]{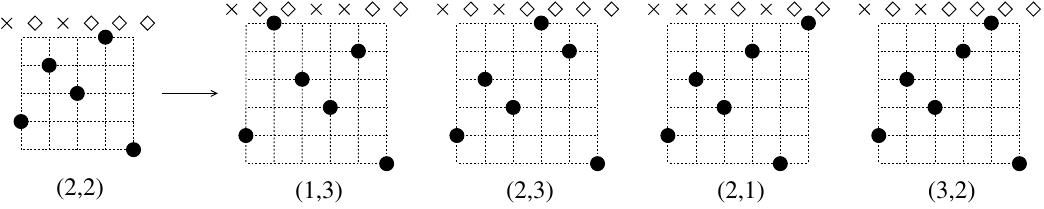}
\end{center}
\caption{The growth of a permutation $\sigma\in\myS$ according to rule~\eqref{rule:newSchroeder}. The $i$th entry $\sigma_i$ is plotted in the grid at coordinate $(i,\sigma_i)$. 
The active (resp. non-active) sites are indicated by the symbol $\Diamond$ (resp. $\times$).
\label{fig:growthSchroederSlicings}}
\end{figure}

\begin{proof}
First, note that if $\sigma \in \myS$, then the permutation obtained by removing the maximal element of $\sigma$ also belongs to $\myS$. 
So we can make permutations of $\myS$ grow by insertion of the maximum.

Second, observe that $\myS$ is a subset of $\Bax$. So the set of active sites (\emph{i.e.}~positions where the new maximum can be inserted while remaining in the class)
is a subset of the set of active sites in the growth of Baxter permutations according to rule~\eqref{rule:BaxterMBM}. 
These active sites are described in~\cite{mbm} and are: 
\begin{itemize}
 \item the sites immediately to the right of right-to-left maxima, and 
 \item the sites immediately to the left of left-to-right maxima.
\end{itemize}
In particular, the two sites surrounding the current maximum are always active. 

\medskip

We claim that the active sites of $\sigma \in \myS$ are the following, where $n$ denotes the size of $\sigma$: 
\begin{itemize}
 \item the sites immediately to the right of right-to-left maxima, and 
 \item for any left-to-right maximum $\sigma_i$, the site immediately to the left of $\sigma_i$, 
 provided that the sequence $\sigma_{i+1} \ldots \sigma_n$ contains no pattern $212^+$ where $2$ is mapped to a value larger than $\sigma_i$.  
\end{itemize}
More formally, the condition above on $\sigma_{i+1} \ldots \sigma_n$ is expressed as follows:  
there is no subsequence $\sigma_a \sigma_b \sigma_c$ of $\sigma_{i+1} \ldots \sigma_n$ such that $\sigma_a > \sigma_i$, $\sigma_b < \sigma_a$ and $\sigma_c = \sigma_a +1$. 

For the first item, it is enough to notice that the insertion of $n+1$ 
in any site located to the right of $n$ cannot create a $41323^+$ or $42313^+$ pattern 
(if it would, then $n$ instead of $n+1$ would give a forbidden pattern in $\sigma$). 

For the second item, consider a left-to-right maximum $\sigma_i$. 
The insertion of $n+1$ immediately to the left of $\sigma_i$ creates a $41323^+$ or $42313^+$ pattern 
if and only if it creates such a pattern where $n+1$ is used as the $4$. 

Assume first that the sequence $\sigma_{i+1} \ldots \sigma_n$ contains a pattern $212^+$ where $2$ is mapped to a value larger than $\sigma_i$. 
Then together with $n+1$ and $\sigma_i$, we get a $41323^+$ or $42313^+$ pattern: such insertions do not produce a permutation in $\myS$. 

On the other hand, assume that the sequence $\sigma_{i+1} \ldots \sigma_n$ contains no pattern $212^+$ where $2$ is mapped to a value larger than $\sigma_i$. 
If the insertion of $n+1$ immediately to the left of $\sigma_i$ creates a $41323^+$ or $42313^+$ pattern, say $(n+1) \sigma_a \sigma_b \sigma_c \sigma_d$, 
then $\sigma_b \sigma_c \sigma_d$ is a $212^+$ pattern in $\sigma_{i+1} \ldots \sigma_n$, and by assumption $\sigma_b < \sigma_i$.
This implies that $\sigma_i$ is larger than all of $\sigma_a$, $\sigma_b$, $\sigma_c$ and $\sigma_d$, 
so that $\sigma_i \sigma_a \sigma_b \sigma_c \sigma_d$ is a $41323^+$ or $42313^+$ pattern in $\sigma$, 
contradicting the fact that $\sigma \in \myS$. 
In conclusion, under the hypothesis that the sequence $\sigma_{i+1} \ldots \sigma_n$ contains no pattern $212^+$ where $2$ is mapped to a value larger than $\sigma_i$, 
then the insertion of $n+1$ immediately to the left of $\sigma_i$ produces a permutation in $\myS$. 

To any permutation $\sigma$ of $\myS$, associate the label $(h,k)$ 
where $h$ (resp.~$k$) denotes the number of active sites to the left (resp.~right) of $\sigma$'s maximum. 
Of course, the permutation $1$ has label $(1,1)$. 
We shall now see that the permutations produced by inserting a new maximum in $\sigma$ have the labels indicated by rule~\eqref{rule:newSchroeder}, 
concluding our proof of Theorem~\ref{thm:newSchroederPerms}. 

Denote by $n$ the size of $\sigma$. 
When inserting $n+1$ in the $i$-th active site (from the left) on the left of $n$, 
this increases by $1$ the number of right-to-left maxima. 
Moreover, no pattern $212^+$ is created, so that all sites to the left of $n$ that were active remain so, provided they remain left-to-right maxima. 
The permutations so produced therefore have labels $(i,k+1)$ for $1\leq i \leq h$. 
Similarly, when inserting $n+1$ immediately to the right of $n$, no $212^+$ is created, and the subsequent permutation has label $(h+1,k)$. 
On the contrary, when inserting $n+1$ to the right of a right-to-left maximum $\sigma_j \neq n$, a pattern $212^+$ is created (as $n \sigma_j (n+1)$). 
Consequently, there are only two left-to-right maxima such that there is no pattern $212^+$ after them with a $2$ of a larger value: 
namely, those are $n$ and $n+1$. 
If $\sigma_j$ was the $i$-th right-to-left maximum of $\sigma$, starting their numbering from the right, 
then the resulting permutation has label $(2,i)$, for any $1\leq i<k$. 
\end{proof}

\subsection{A Schr\"oder family of mosaic floorplans}

\emph{Mosaic floorplans} (a simplified version of general floorplans) were defined by Hong et al.~\cite{Hong} in the context of chip design. 
A mosaic floorplan is a rectangular partition of a rectangle by means of segments that do not properly cross, 
\emph{i.e.}~every pair of segments that intersect forms a T-junction of type \tikz[scale=0.15]{
\draw (0,0) -- (2,0);
\draw (1,0) -- (1,1);
}, \tikz[scale=0.15]{
\draw (0,1) -- (2,1);
\draw (1,1) -- (1,0);
}, \tikz[scale=0.15]{
\draw (0,0) -- (0,2);
\draw (0,1) -- (1,1);
}, or \tikz[scale=0.15]{
\draw (1,1) -- (2,1);
\draw (2,0) -- (2,2);
}. 
The empty spaces between the segments are called \emph{rooms}. 
\emph{Internal segments} of a mosaic floorplan $F$ are segments that are not part of the bounding rectangle of $F$. 
Mosaic floorplans are considered up to equivalence under the action of sliding segments, namely up to translating their internal segments 
to modify the dimensions of the rooms yet without removing any T-junction.
Figure~\ref{fig:PFP} shows two mosaic floorplans that are equivalent. 
From now on, we write mosaic floorplan to denote an \emph{equivalence class} of mosaic floorplans. 
So, the two objects of Figure~\ref{fig:PFP} are rather two representatives of the same mosaic floorplan. 
Yao et al.~\cite{Yao} proved that mosaic floorplans are enumerated by Baxter numbers. 

In this section, we explain the growth of mosaic floorplans according to rule~\eqref{rule:BaxterMBM}, 
\emph{i.e.}~along the generating tree $\TBax$. 
Then, we define a subfamily of mosaic floorplans enumerated by Schr\"oder numbers, which we call \emph{Schr\"oder floorplans}. 
We prove that they grow according to rule~\eqref{rule:newSchroeder}. 
% the restriction to the canonical subtree $\TSch$ of $\TBax$ allows to define a subfamily of mosaic floorplans enumerated by Schr\"oder numbers, 
% which we call \emph{Schr\"oder floorplans}. 
% Unlike in the case of permutations discussed above, Schr\"oder floorplans can be characterized independently of generating trees, 
% by forbidden configurations of segments -- see Definition~\ref{dfn:SchroederPFP}.

\begin{remark}
In their article~\cite{Yao}, Yao et al.~have also described a subfamily of mosaic floorplans enumerated by Schr\"oder numbers, called \emph{slicing floorplans}. 
They are defined by the avoidance of the configurations 
$\begin{array}{c}
\tikz[scale=0.25]{
\draw (1,0) -- (1,2);
\draw (0,2) -- (2,2);
\draw (2,1) -- (2,3);
\draw (1,1) -- (3,1);}
\end{array}$ and 
$\begin{array}{c}
\tikz[scale=0.25]{
\draw (6,1) -- (6,3);
\draw (6,2) -- (8,2);
\draw (7,0) -- (7,2);
\draw (5,1) -- (7,1);
}\end{array}$. 

Our Schr\"oder floorplans are also defined by a forbidden configuration of segments -- see Definition~\ref{dfn:SchroederPFP}. 
However, slicing floorplans do not coincide with our Schr\"oder floorplans. 
Nevertheless, both slicing floorplans and Schr\"oder floorplans avoid the configuration $\begin{array}{c}
\tikz[scale=0.25]{
\draw (1,0) -- (1,2);
\draw (0,2) -- (2,2);
\draw (2,1) -- (2,3);
\draw (1,1) -- (3,1);}
\end{array}$,
and the similarity of the forbidden configurations is striking. 
We leave open the problem of explaining this similarity combinatorially, 
for instance by describing an explicit bijection between slicing floorplans and Schr\"oder floorplans. 

Note that we were not able to describe a growth of slicing floorplans that follows rule~\eqref{rule:newSchroeder}.
\end{remark}

A difficulty in working with mosaic floorplans is that they are equivalence classes of combinatorial objects. 
To address this difficulty, \emph{packed floorplans} have been introduced in~\cite{PackFP}, 
where it is proved that every mosaic floorplan contains exactly one packed floorplan. 
(In some sense, packed floorplans can then be considered as canonical representatives of mosaic floorplans.) 
It follows from the enumeration of mosaic floorplans in \cite{Yao} that packed floorplans are enumerated by Baxter numbers. 

\begin{definition}\label{PFP}
A packed floorplan (PFP) of dimension $(d,\ell)$ is a
partition of a rectangle of width $\ell$ and height $d$, 
by means of segments that do not properly cross, 
into $d+\ell-1$ rectangular blocks whose sides have
integer lengths and such that the pattern \tikz[scale=0.15]{
\draw (0,2) -- (1,2) -- (1,3);
\draw (2,0) -- (2,1) -- (3,1);
} is avoided, \emph{i.e.}~for every
pair of blocks $(b_1, b_2)$, denoting $(x_1,y_1)$ the coordinates of the bottom rightmost
corner of $b_1$ and $(x_2, y_2)$ those of the top leftmost corner of $b_2$, it is not possible
to have both $x_1\leq x_2$ and $y_1\geq y_2$.

The size of a packed floorplan of dimension $(d,\ell)$ is $n=d+\ell-1$ and 
the set of packed floorplans of size $n$ is denoted $\mathcal{F}_n$.
\end{definition}

\begin{observation} \label{obs:PFP}
Some properties of PFPs have been proved in~\cite{PackFP}. 
Relevant to our work is the fact that every horizontal (resp. vertical) line of integer coordinate inside the bounding rectangle of a PFP 
is the support of exactly one segment of the PFP.
\end{observation}
 
Notice the slight change of terminology: 
while we speak of \emph{rooms} in mosaic floorplans (whose dimensions may change inside an equivalence class), 
we prefer the word \emph{blocks} in PFPs (since the dimensions of the empty spaces between segments of a PFP may not change). 

Figure~\ref{fig:PFP}$(a)$ shows an example of a packed floorplan, 
while Figure~\ref{fig:PFP}$(b)$ shows another (non-packed) representative of the same mosaic floorplan. 

\begin{figure}[ht]
\begin{center}
\includegraphics[scale=0.35]{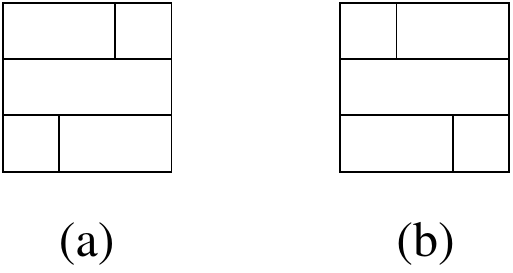}
\end{center}
\caption{(a) An example of a packed floorplan of dimension $(3,3)$, 
(b) a non-packed representative of the same mosaic floorplan. 
\label{fig:PFP}} 
\end{figure}

\begin{theorem}
\label{thm:PFP}
The family of PFPs grows according to the succession rule~\eqref{rule:BaxterMBM}, \emph{i.e.}~along the generating tree $\TBax$.
\end{theorem}

Observe that a generating tree for PFPs is presented in \cite{PackFP} (via a procedure called \emph{InsertTile} for adding a new block in PFPs). 
Considering only the first few levels of this generating tree, it is immediately clear that it is not isomorphic to $\TBax$. 
Therefore, to prove Theorem~\ref{thm:PFP}, we need to define a new way of adding a block to a PFP. 

\begin{proof}
Consider a PFP $F$ of dimension $(d,\ell)$ and size $n=d+\ell-1$. 
We give to $F$ the label $(h,k)$, where 
$h$ (resp.~$k$) is one greater than the number of internal segments of $F$ 
that meet the right (resp.~upper) border of the bounding rectangle of $F$. 
We build $h+k$ children of size $n+1$ for $F$ as described below. 
(See also Figure~\ref{fig:PPgrowth}, which shows an example of the growth of PFPs of dimension $(3,3)$ having label $(3,2)$.)

The first $h$ children, of dimension $(d,\ell+1)$, are obtained by adding a new block $b$ on the right of the north-east corner of $F$: 
the left side of $b$ then forms a new internal segment that can reach the bottom border of the floorplan 
or stop when meeting any segment $s$ incident with the right border of $F$ (note that there are $h-1$ such segments). 
The segments reaching the right border of $F$ which are below $s$ (and the corresponding blocks) 
are then extended to reach the right border of the wider rectangle of dimension $(d,\ell+1)$. 

The other $k$ children, of dimension $(d+1,\ell)$, are obtained by adding a new block $b$ on top of the north-east corner of $F$: 
similarly, the bottom side of $b$ then forms a new internal segment that can reach the left border of the floorplan 
or stop when meeting any segment $s$ incident with the upper border of $F$ (note that there are $k-1$ such segments). 
Again, the segments reaching the upper border of $F$ which are to the left of $s$ (and the corresponding blocks) 
are extended to reach the upper border of the higher rectangle of dimension $(d+1,\ell)$. 

With $h$ and $k$ defined as above, and giving label $(h,k)$ to PFPs, 
it is clear that the children of a PFP with label $(h,k)$ have labels 
$(i,k+1)$ for $1\leq i \leq h$ (insertion of a new block on the right of $F$) 
and $(h+1,j)$ for $1\leq j \leq k$ (insertion of a new block on top of $F$). 
Moreover, the unique packed floorplan of size $1$ (having dimension $(1,1)$) has no internal segment, so its label is $(1,1)$. 

To prove that PFPs grow according to rule~\eqref{rule:BaxterMBM}, 
it is then enough to show that the above construction generates each PFP exactly once. 

First, we prove by induction that this construction generates only PFPs. 
The relation between the number of blocks and the dimensions of the bounding rectangle is clearly satisfied. 
So we only need to check that, if $F$ is a PFP, then all of its children avoid the pattern 
\tikz[scale=0.13]{
\draw (0,2) -- (1,2) -- (1,3);
\draw (2,0) -- (2,1) -- (3,1);
}. 
Consider a child $F'$ of $F$ obtained by adding a new block $b$ on the right of the north-east corner of $F$. 
The bottom right corners of the existing blocks 
may only be modified by being moved to the right. 
Similarly, if a child $F'$ of $F$ is obtained by adding a new block $b$ on top of the north-east corner of $F$, 
the top left corners of the existing blocks 
may only be modified by being moved upwards. 
So, in either case, those corners cannot create any pattern 
\tikz[scale=0.13]{
\draw (0,2) -- (1,2) -- (1,3);
\draw (2,0) -- (2,1) -- (3,1);
}. 
And the new block $b$ cannot create any such pattern either, 
since it has no block above it nor to its right. 

Next, we prove by induction that all PFPs are generated. 
Consider a PFP $F$ of size $n\geq 2$. 
Let $b$ be the block in the north-east corner of $F$ and $s$ (resp.~$t$) be the left (resp.~bottom) side of $b$. 
Their graphical configurations can be either $\begin{array}{c}\tikz[scale=0.3]{
\draw (1,0) rectangle (3,1) node[left=14pt] {$_s$} node[below=6pt] {$_t$};
\draw (0.5,0) -- (1,0);
}
\end{array}$ or $\begin{array}{c}\tikz[scale=0.3]{
\draw (1,0) rectangle (3,1) node[left=14pt] {$_s$} node[below=6pt] {$_t$};
\draw (1,-0.5) -- (1,0);
}\end{array}$. 

We claim that in the first (resp.~second) case, $b$ has width (resp.~height) $1$. 
This follows from Observation~\ref{obs:PFP}. 
Indeed, assuming it is not the case, we can consider the segment of $F$ lying on the rightmost internal vertical line 
(resp. the topmost internal horizontal line), and display an occurrence of the pattern \tikz[scale=0.15]{
\draw (0,2) -- (1,2) -- (1,3);
\draw (2,0) -- (2,1) -- (3,1);
}, which is forbidden in PFPs -- see (a) and (b) of Figure~\ref{fig:removingPFP}. 

\begin{figure}[ht]
\begin{center}
\includegraphics[scale=0.3]{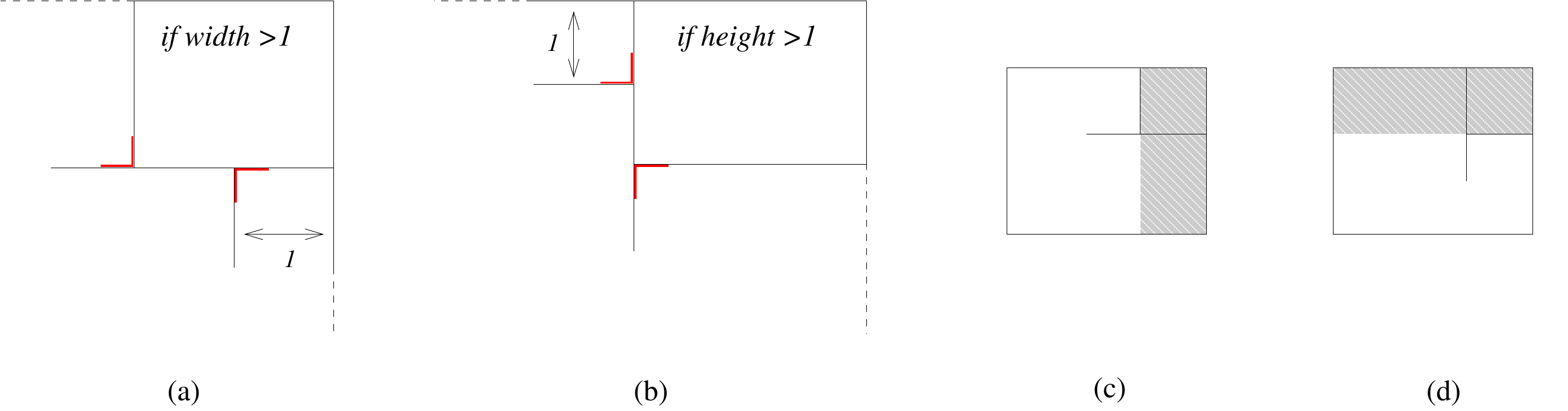}
\end{center}
\caption{Proof of Theorem~\ref{thm:PFP}. (a) and (b): The block $b$ has width (resp.~height) $1$. (c) and (d): Removing the hatched part from the PFP $F$ produces the PFP $F'$.}
\label{fig:removingPFP}
\end{figure}

So, in the first (resp.~second) case, we define $F'$ by deleting the part of $F$ on the right of the line on which $s$ lies, 
(resp.~the part of $F$ above the line on which $t$ lies), which is at distance $1$ from the boundary of the rectangle -- see (c) and (d) of Figure~\ref{fig:removingPFP}. 
This removal does not create any occurrence of the forbidden pattern. 
So $F'$ is indeed a PFP, of size one less than $F$, 
and $F$ is by construction one of the children of $F'$. 

Finally, it remains to prove that no PFP is generated several times. 
Obviously, the children of a given PFP are all different. 
So we only need to make sure that the parent of a PFP $F$ is uniquely determined. 
Looking again at the block $b$ in the north-east corner of $F$, 
and at the type of the T-junction at the bottom-left corner of $b$, 
we determine whether $b$ was added on top or on the right of the north-east corner of its parent. 
By construction, the parent is then uniquely determined: it is necessarily obtained from $F$ by deleting the parts of $F$ described above. 
\end{proof}

\begin{figure}[ht]
\begin{center}
\includegraphics[scale=0.35]{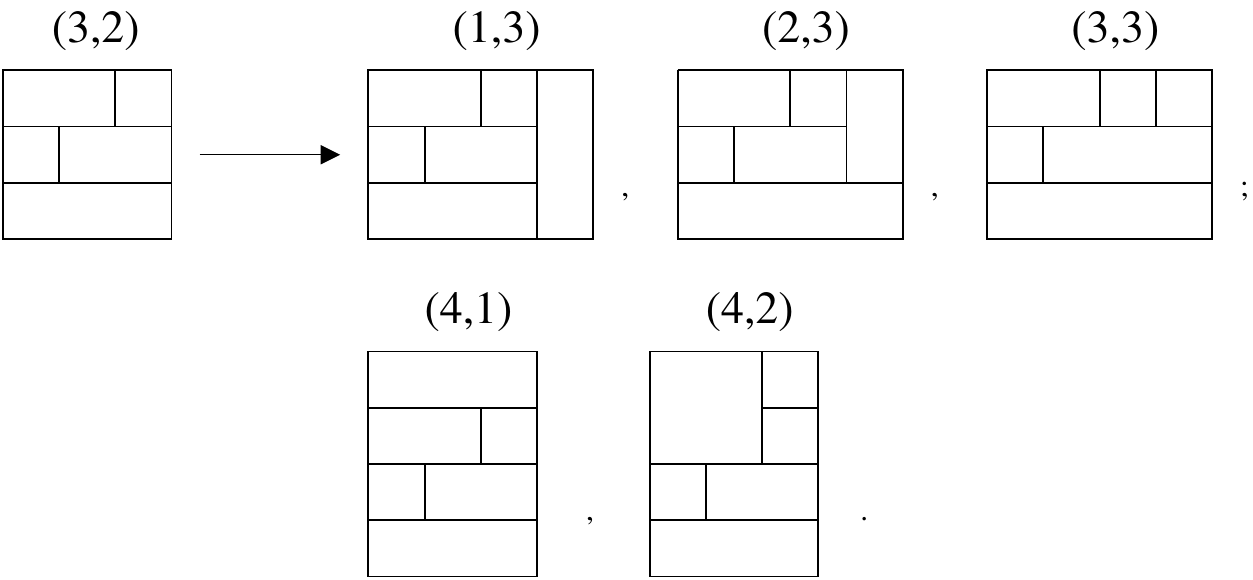}
\end{center}
\caption{The growth of packed floorplans following rule~\eqref{rule:BaxterMBM}. 
\label{fig:PPgrowth}} 
\end{figure}

\begin{definition}
\label{dfn:SchroederPFP}
A Schr\"oder PFP is a PFP as in Definition~\ref{PFP}, whose internal segments avoid the following configuration: 
$$\begin{array}{c}
\tikz[scale=0.25]{
\draw (0,1) -- (0,2);
\draw (-1,2) -- (2,2);
\draw (2,0) -- (2,3);
\draw (1,0) -- (3,0);
}
\end{array}.$$
\end{definition}
\noindent (Note that an occurrence of the above configuration where the bottom segment is the bottom border 
-- which is of course not an internal segment -- does not prevent a PFP from being a Schr\"oder PFP.)

Figure~\ref{fig:SPF2} shows some packed floorplans that contain the forbidden configuration of Definition~\ref{dfn:SchroederPFP} and so, they are not Schr\"oder PFPs.

\begin{figure}[ht]
\begin{center}
\includegraphics[scale=0.35]{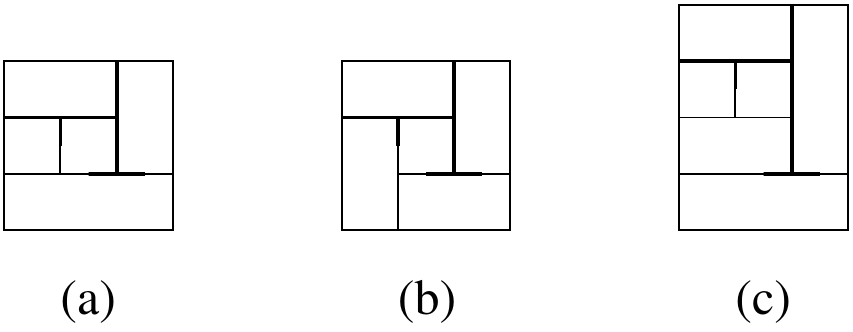}
\end{center}
\caption{(a),(b) The two packed floorplans of size $5$ which are not Schr\"oder PFPs, 
(c) a non-Schr\"oder packed floorplan of size $6$.
\label{fig:SPF2}}
\end{figure}

\begin{theorem}
\label{thm:SchPFP}
The generating tree obtained by letting Schr\"oder PFPs grow by insertion of a new block as in the proof of Theorem~\ref{thm:PFP} is $\TSch$. 
More precisely, they grow following rule~\eqref{rule:newSchroeder}.
\end{theorem}

\begin{proof}
Let $F$ be a PFP, and $b$ be the block in the north-east corner of $F$. 
Recall that the parent $F'$ of $F$ was described in the proof of Theorem~\ref{thm:PFP}. 
It follows immediately that if $F$ is a Schr\"oder PFP, 
then $F'$ is also a Schr\"oder PFP. 
Consequently, we can make Schr\"oder PFPs grow by addition of a new block either on the right of the north-east corner or above it, 
as in the proof of Theorem~\ref{thm:PFP}. 

\medskip

Let $F$ be a Schr\"oder PFP. 
We consider all its children following the growth of PFPs described in the proof of Theorem~\ref{thm:PFP}, 
and we determine which of them are Schr\"oder PFPs. 
Let $b$ be a new block added to $F$. 
Note first that the addition of $b$ may only create forbidden configurations involving the sides of $b$. 
Moreover, if such a forbidden configuration is created, the sides of $b$ are necessarily the segments shown in bold line in the following picture: 
$\begin{array}{c}
\tikz[scale=0.25]{
\draw (0,1) -- (0,2);
\draw (-1,2) -- (2,2);
\draw[very thick] (2,0) -- (2,3);
\draw (1,0) -- (2,0);
\draw[very thick] (2,0) -- (3,0);}
\end{array}$.
In particular, the T-junction at the bottom left corner of $b$ is of type \tikz[scale=0.15]{
\draw (0,0) -- (2,0);
\draw (1,0) -- (1,1);
}.

If $b$ is added above the north-east corner of $F$, 
then by construction the bottom side of $b$ reaches the left border of $F$ or forms a T-junction of type \tikz[scale=0.15]{
\draw (0,0) -- (0,2);
\draw (0,1) -- (1,1);
} with a segment meeting the upper border of $F$. 
So the forbidden configurations cannot be created, 
and all PFPs obtained by adding blocks above the north-east corner of $F$ are Schr\"oder PFPs. 

On the contrary, if $b$ is added on the right of the north-east corner of $F$, 
then the T-junction at the bottom left corner of $b$ is of type \tikz[scale=0.15]{
\draw (0,0) -- (2,0);
\draw (1,0) -- (1,1);
}, so a forbidden configuration may be created. 
More precisely, the forbidden configuration is generated if and only if the following situation occurs: 
the segment corresponding to the left side of $b$ reaches an internal segment meeting the right border of $F$, which in turn
is below another internal segment that is incident with the right border of $F$ 
and that forms a T-junction of type  \tikz[scale=0.15]{
\draw (0,1) -- (2,1);
\draw (1,1) -- (1,0);
} with some internal segment. 
So, to determine which children of $F$ are Schr\"oder PFPs, 
among those obtained by adding $b$ on the right of the north-east corner of $F$, 
it is essential to identify the topmost internal segment  
which meets the right border of $F$ and which forms a T-junction of type \tikz[scale=0.15]{
\draw (0,1) -- (2,1); %node[right] {$_p$};
\draw (1,1) -- (1,0);
} with some internal segment of $F$. 
If such a segment exists, it is denoted $p_F$.
Then, adding $b$ to $F$, a Schr\"oder PFP is obtained exactly when the bottom side of $b$ is 
either the bottom border of $F$ 
or an internal segment meeting the right border of $F$ which is above $p_F$ ($p_F$ included). 

\medskip

With the above considerations, it is not hard to prove that Schr\"oder PFPs grow according to rule~\eqref{rule:newSchroeder}. 
To any Schr\"oder PFP $F$, we assign the label $(h,k)$ 
where 
\begin{itemize}
 \item if $p_F$ exists, $h$ is one greater than the number of internal segments meeting the right border of $F$ above $p_F$ (included), 
 \item if $p_F$ does not exist, $h$ is one greater than the total number of internal segments meeting the right border of $F$, 
 \item in both cases, $k$ is one greater than the number of internal segments meeting the upper border of $F$.
\end{itemize}
Of course, the only (Schr\"oder) PFP of size $1$ has label $(1,1)$. 
Following the growth previously described, a Schr\"oder PFP $F$ of label $(h,k)$ produces:
\begin{itemize}
\item $h$ Schr\"oder PFPs obtained by adding a block $b$ on the right of the north-east corner of $F$. 
The left side of $b$ may reach the bottom border of $F$, and then a Schr\"oder PFP of label $(1,k+1)$ is obtained. 
It may also reach any internal segment $s$ incident with the right border of $F$ that is above $p_F$ (included), 
and Schr\"oder PFPs of labels $(2,k+1),\ldots,(h,k+1)$ are obtained in this way. 
Indeed, for the PFP $F'$ produced, $p_{F'} = p_F$, but the number of internal segments meeting the right border above it takes a value between $1$ and $h-1$. 
\item $k$ Schr\"oder PFPs obtained by adding a block $b$ above the north-east corner of $F$. 
The bottom side of $b$ may reach the rightmost segment incident with the upper border of $F$, and then a Schr\"oder PFP of label $(h+1,k)$ is obtained. 
But if it reaches any other segment incident with the upper border of $F$ (left border of $F$ included), then a T-junction of type \tikz[scale=0.15]{
\draw (0,1) -- (2,1);
\draw (1,1) -- (1,0);
} is formed with at least one internal segment meeting the upper border of $F$. 
By definition, for the Schr\"oder PFP $F'$ produced, we therefore have that $p_{F'}$ 
is the segment that supports the bottom edge of $b$. 
Consequently, the labels of the Schr\"oder PFPs produced are $(2,k-1),\ldots,(2,1)$. 
\end{itemize}
This concludes the proof that Schr\"oder PFPs grow with rule~\eqref{rule:newSchroeder}, 
and so along the generating tree $\TSch$.
\end{proof}

To illustrate the growth of Schr\"oder PFPs with rule~\eqref{rule:newSchroeder}, 
note that, seen as a Schr\"oder PFP, the object whose growth is depicted in Figure~\ref{fig:PPgrowth} has label $(2,2)$ 
and it has only four children (the middle object of the first line is not produced, and indeed it is not a  Schr\"oder PFP). 

Figure~\ref{fig:SPPgrowth} shows an example of the growth of a Schr\"oder PFP $F$ of dimension $(4,2)$ having label $(3,1)$. 
The segment $p_F$ (the topmost internal segment of $F$ which meets the right border and forms a T-junction of type \tikz[scale=0.15]{
\draw (0,1) -- (2,1); %node[right] {$_p$};
\draw (1,1) -- (1,0);
} with an internal segment of $F$) is highlighted in bold line. 

\begin{figure}[ht]
\begin{center}
\includegraphics[scale=0.35]{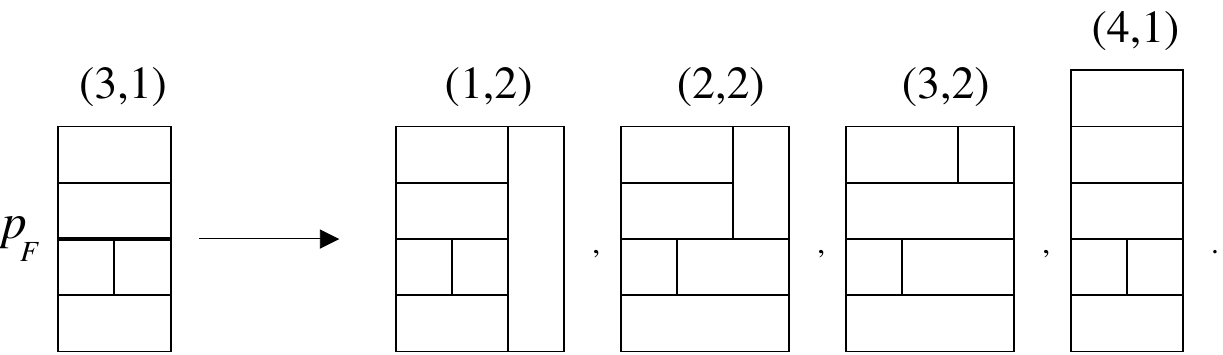}
\end{center}
\caption{The growth of a Schr\"oder PFP $F$ following rule~\eqref{rule:newSchroeder}. 
\label{fig:SPPgrowth}} 
\end{figure}

\begin{remark}
In the same fashion, we can define a subfamily of PFP enumerated by the Catalan numbers, 
and prove that they grow according to rule~\eqref{rule:Catalan}
(the label $k$  of a Catalan PFP is one greater than the number of internal segments meeting its upper border). 
A Catalan PFP would be a PFP as in Definition~\ref{PFP}, whose internal segments avoid the configuration \tikz[scale=0.25]{
\draw (0,0) -- (2,0);
\draw (1,0) -- (1,1);
}. 
The proof that they grow according to rule~\eqref{rule:Catalan} is omitted, 
but very similar to that of Theorem~\ref{thm:SchPFP}. 

We point out that a different proof that these Catalan PFP (considered up to a $-90^\circ$ rotation) are enumerated by the Catalan numbers 
has been given in \cite[Example  $\mathbb{F}_{13}$ in the Appendix]{Brak}, 
giving a product rule for Catalan PFP which corresponds to the standard Catalan product rule.
\end{remark}

\section{More families of restricted slicings}
\label{sec:extensions}

With the Schr\"oder slicings, we have seen in Section~\ref{sec:Schroeder} one way of specializing the succession rule~\eqref{rule:BaxterMBM}. 
In this section, we are interested in other specializations of rule~\eqref{rule:BaxterMBM}, which allow us to define Catalan slicings, 
$m$-skinny slicings and $m$-row-restricted slicings, for any integer $m \geq 0$. 
The next section will explore the properties of the generating functions for $m$-skinny slicings and $m$-row-restricted slicings. 

\subsection{Catalan slicings}

Similarly to the path followed to define Schr\"oder slicings, 
we can consider the generating tree $\TBax$ of Baxter slicings, and its subtree isomorphic to $\TCat$ discussed in Subsection~\ref{subsec:Cat_in_Bax}, 
to define ``Catalan slicings'' of parallelogram polyominoes. 
As expected, we find exactly one Catalan slicing $C$ for every parallelogram polyomino $P$, 
namely, the Baxter slicing of shape $P$ whose horizontal blocks all have width $1$. 
Alternatively, $C$ can be recursively described as follows: 
if the top row of $P$ contains just one cell, then this cell constitutes a horizontal block of $C$, 
and we proceed by computing the Catalan slicing of $P$ minus this top row; 
otherwise, the rightmost column of $P$ constitutes a vertical block of $C$, 
and we proceed by computing the Catalan slicing of $P$ minus this rightmost column. 

\subsection{Skinny slicings}

We have seen in Definition~\ref{dfn:SchroederSlicing} that Schr\"oder slicings are defined by condition~\eqref{eq:lr1}, that is to say, $\ell(u) \leq r(u)+1,$ for any horizontal block $u$.
Figure~\ref{fig:SchroederSlicings}(a) (page~\pageref{fig:SchroederSlicings}) shows which quantities are to be checked for satisfying the above condition. A rough idea to characterize a Schr\"oder slicing of a parallelogram polyomino $P$ is: every corner of the lower path defining $P$ may have above it only horizontal blocks that do not protrude more than \emph{one cell} leftward its $x$-coordinate.

Therefore, this condition~\eqref{eq:lr1} can be naturally generalized for any integer $m\geq0$: for any horizontal block $u$, 
\begin{equation}
\ell(u) \leq r(u) +m. \tag{$\mathnormal{\ell r_m}$} \label{eq:lrm}
\end{equation}

\begin{definition}
An \emph{$m$-skinny slicing} is a Baxter slicing such that for any horizontal block $u$, the inequality~\eqref{eq:lrm} holds.
\label{dfn:SkinnySlicing}
\end{definition}

Note that an $m'$-skinny slicing, with $m'\leq m$, is an $m$-skinny slicing as well. 
For instance, the slicing of Figure~\ref{fig:SchroederSlicings}(a) is an $m$-skinny slicing, for any $m\geq3$.

\begin{theorem}
A generating tree for $m$-skinny slicings is described by the following succession rule:

\begin{equation}
\textrm{root labeled } (1,1) \textrm{ \quad and \quad} (h,k) \rightsquigarrow \begin{cases}
(1,k+1) , (2,k+1) , \ldots  , (h,k+1),\vspace{0.2cm}\\ 
(h+1, 1), \ldots , (h+1,k-1),(h+1, k), &\mbox{ if } h<m, \\
(m+1,1) , \ldots , (m+1,k-1) , &\mbox{ if } h\geq m,\\
(h+1,k). &\mbox{ if } h\geq m.
\end{cases} \tag{$\Omega_m$}\label{rule:m_slicings}
\end{equation}
\label{thm:SkinnySlicings_growth}
\end{theorem}

\begin{proof}
The proof follows the exact same steps as the proof of Theorem~\ref{thm:SchroederSlicings_growth}, which corresponds to $m=1$. 
The only difference is that the maximal width of the horizontal block that may be added in the third case is $\min(h+1,m+1)$ instead of $2$. 
\end{proof}

Considering the case $m=0$, we obtain a family of Baxter slicings which is intermediate between Catalan slicings (for which $\ell(u) = 1$, for all horizontal blocks $u$) 
and Schr\"oder slicings (\emph{i.e.}~$1$-skinny slicings). 
The first few terms of the enumeration sequence of $0$-skinny slicings are $1,2,6,21,80,322,\ldots$. This sequence, and a curious enumerative result relating to it, are further explored in Section~\ref{sec:gen_funcs} (see Theorem~\ref{thm:2-row-restricted}).

\subsection{Row-restricted slicings} \label{paragraph:row-restricted}

Conditions~\eqref{eq:lrm} naturally generalize the condition that defines Schr\"oder slicings, 
but it is not the most natural restriction on horizontal blocks of Baxter slicings one may think of. 
Indeed, for some parameter $m\geq 1$, we could simply impose that horizontal blocks have width no larger than $m$. 
In what follows, we study these objects under the name of \emph{$m$-row-restricted slicings}. 

Note that, taking $m=1$, we recover Catalan slicings, 
and that the case $m=0$ is degenerate, since there is only one $0$-row-restricted slicing of any given size: 
the horizontal bar of height $1$ and width $n$ divided in (vertical) blocks made of one cell only. 

\begin{theorem}
A generating tree for $m$-row-restricted slicings is described by the succession rule:
\begin{equation}
\textrm{root labeled } (1,1) \textrm{\quad and \quad } (h,k) \rightsquigarrow 
\begin{cases}
(1, k + 1), (2, k + 1), \ldots , (h, k + 1), 
\vspace{0.2cm} \\
(h+1,1) , (h+1,2) , \ldots , (h+1,k), &\mbox{ if } h< m \\
(m,1) , (m,2) , \ldots ,  (m,k). &\mbox{ if } h= m\\
\end{cases}  \label{rule:row_restricted} \tag{$\Upsilon_m$}
\end{equation}
\label{thm:growth_row_restricted} 
\end{theorem}

\begin{proof}
Again, the proof is similar to those of Theorems~\ref{thm:growth_Baxter_slicings} and~\ref{thm:SkinnySlicings_growth}, 
and when a slicing has label $(h,k)$, $h$ (resp.~$k$) indicates 
the maximal width of a horizontal block that may be added 
(resp.~the maximal height of a vertical block that may be added). 
In the case of $m$-row-restricted slicings, when a vertical block is added to the right, 
the maximal width of a horizontal block that may be added afterward increases by $1$, 
except if it was $m$ already, in which case it stays at $m$. 
\end{proof}

%When it comes to $m$-row-restricted slicings for general $m$, 
%the same questions as for $m$-skinny slicings can be asked: 
%can their generating functions be computed from the system of functional equations derived from rule~\eqref{rule:row_restricted}? 
%or can the nature of these generating functions be derived? 

%%%%%%%%%% 
%
%\subsection{Other extensions}
%
%We believe our results and the questions left open demonstrate that 
%slicings of parallelogram polyominoes have a rich combinatorics yet to explore. 
%We hope to contribute to this study in future work. 
%We also ask whether similar interesting phenomena may happen when ``slicing'' other families of polyominoes. 

\section{Generating functions and functional equations}\label{sec:gen_funcs}

Recall that a univariate function $f(x)$ is \emph{algebraic} if there exists a polynomial $P(x,y)$ such that $y=f(x)$ is a root of $P(x,y)=0$; 
while $f(x)$ is \emph{D-finite} if it is the solution of a linear differential equation $c_m(x)f^{(m)}(x) + c_{m-1}(x)f^{(m-1)}(x) + \ldots + c_0(x)f(x) = 0$, where all the $c_i(x)$ are polynomials.
Note also that every algebraic function is D-finite.

%\subsection{Catalan, Schr\"oder and Baxter}

Examples of algebraic generating functions are given by the well-known generating functions for Catalan and Schr\"oder numbers:
\begin{align}
F_\text{Cat}(x) &= \frac{1-\sqrt{1-4x}}{2x} \tag{GF$_\text{Cat}$} \label{eqn:catalan_gf}\\
F_\text{Sch}(x) &= \frac{1-x-\sqrt{1-6x+x^2}}{2x} \tag{GF$_\text{Sch}$} \label{eqn:schroder_gf} %\\
%F_\text{B}(x) &= \begin{multlined}[t] -1 + \frac{1}{3x^2}\left[-1+x+(1-2x){}_2F_1\left(-\frac23,\frac23;1;\frac{27x^2}{(1-2x)^3}\right) \right. \\ \left.+\frac{x(1+11x-8x^2)}{(1-2x)^2}{}_2F_1\left(\frac13,\frac23;2;\frac{27x^2}{(1-2x)^3}\right)\right].\end{multlined} \tag{GF$_\text{Bax}$}
\end{align}

On the other hand, the generating function $F_\text{Bax}(x)$ for Baxter numbers, as expressed in \cite{BM}, is D-finite but not algebraic.

\subsection{Functional equations for skinny and row-restricted slicings}

In this subsection we will set out the functional equations satisfied by the generating functions for $m$-skinny slicings and $m$-row-restricted slicings, as defined in Section~\ref{sec:extensions}. The solutions of these functional equations will then be discussed in the following two subsections.

\smallskip

We begin by treating separately the set of $0$-skinny slicings. From Theorem~\ref{thm:SkinnySlicings_growth}, $0$-skinny slicings grow according to rule $(\Omega_0)$:
\begin{equation}
\textrm{root labeled } (1,1) \textrm{ \quad and \quad} (h,k) \rightsquigarrow \begin{cases}
                        (1, k + 1), (2, k + 1)\, \ldots , (h, k + 1),\\
                        (1,1) , (1,2) , \ldots , (1,k-1) , (h+1,k).
                       \end{cases} \tag{$\Omega_0$} \label{rule:0_slicings}
\end{equation}
Now let 
\[F_{0\text{-}Sk}(x;u,v) \equiv F_{0\text{-}Sk}(u,v)=\sum_{\alpha\in {\cal T}_{\Omega_0}}  x^{n(\alpha)} u^{h(\alpha)} v^{k(\alpha)}\]
be the generating function for $0$-skinny slicings, 
where the variable $x$ takes into account the size $n(\cdot)$ of the slicing, while $u$ and $v$ correspond to the labels $h$ and $k$ of the object. 
The rule~\eqref{rule:0_slicings} can be translated into the following functional equation  
\begin{align}
F_{0\text{-}Sk}(u,v) &= xuv +\hspace{-.15cm}\sum_{\alpha\in {\cal T}_{\Omega_0}}\hspace{-.05cm}x^{n+1}(u+\ldots +u^h) v^{k+1} + 
\hspace{-.15cm}\sum_{\alpha\in {\cal T}_{\Omega_0}}\hspace{-.05cm}x^{n+1}u (v+\ldots +v^{k-1})  +\hspace{-.15cm}\sum_{\alpha\in {\cal T}_{\Omega_0}}\hspace{-.05cm}x^{n+1}u^{h+1} v^{k} \notag \\
&= \begin{multlined}[t] xuv+\frac{xuv}{1-u} \left[F_{0\text{-}Sk}(1,v)-F_{0\text{-}Sk}(u,v)\right]+ \frac{xu}{1-v} \left[v F_{0\text{-}Sk}(1,1)-F_{0\text{-}Sk}(1,v)\right]\\
+ xu F_{0\text{-}Sk}(u,v).\end{multlined} \tag{$0$-Sk} \label{eqn:0skinny_func_eqn}
\end{align}
Next, recall that $1$-skinny slicings are exactly Schr\"oder slicings, whose generating function is given by $F_\text{Sch}(x)$ in~\eqref{eqn:schroder_gf}. 

Thereafter, fix some $m\geq 2$. For any $i<m$, let
\[F_i(x;u,v)\equiv F_i(u,v) =\sum_{\alpha} x^{n(\alpha)} u^{h(\alpha)} v^{k(\alpha)}\]
be the trivariate generating function for $m$-skinny slicings 
whose label according to rule~\eqref{rule:m_slicings} is of the form $(i,\cdot)$. 
For $i=m$, $F_m(x;u,v)\equiv F_m(u,v) =\sum_{\alpha} x^{n(\alpha)} u^{h(\alpha)} v^{k(\alpha)}$ 
is defined a bit differently: it is the trivariate generating function for $m$-skinny slicings 
whose label according to rule~\eqref{rule:m_slicings} is of the form $(j,\cdot)$ for any $j \geq m$. 
Note that by definition $F_i(u,v) = u^i F_i(1,v)$ for all $i<m$, but this does not hold for $i=m$. 
The trivariate generating function for $m$-skinny slicings is given by $F_{m\text{-}Sk}(x;u,v)\equiv F_{m\text{-}Sk}(u,v)=\sum_i F_i(u,v)$. 
Similarly to the above case of $0$-skinny slicings the rule~\eqref{rule:m_slicings} translates into the following system: 
\begin{align}
F_1(u,v) &= xuv+xuv\left[F_1(1,v)+F_2(1,v)+\ldots+F_m(1,v)\right] \tag{Sk$_1$} \label{eqn:skinny_F1} \\[-5pt]
%F_2(u,v) &=  \frac{xu^2v}{1-v}\left[F_1(1,1)-F_1(1,v)\right]+xu^2v\left[F_2(1,v)+\ldots+F_m(1,v)\right]\\[-5pt]
&\;\;\vdots \notag \\[-5pt]
F_i(u,v) &= \frac{xu^iv}{1-v}\left[F_{i-1}(1,1)-F_{i-1}(1,v)\right]+xu^iv\left[F_i(1,v)+\ldots+F_m(1,v)\right] \quad \text{for } 1<i<m \tag{Sk$_i$} \label{eqn:skinny_Fi} \\[-5pt]
&\;\;\vdots\notag \\[-5pt]
F_m(u,v) &= \begin{multlined}[t] \frac{xu^mv}{1-v}\left[F_{m-1}(1,1)-F_{m-1}(1,v)\right]+\frac{xu^{m+1}}{1-v}\left[vF_m(1,1)-F_m(1,v)\right]+xuF_m(u,v)\\
+\frac{xuv}{1-u}\left[u^{m-1}F_m(1,v)-F_m(u,v)\right]. \end{multlined} \tag{Sk$_m$} \label{eqn:skinny_Fm}
\end{align}

More precisely, Equation~\eqref{eqn:skinny_F1} is obtained by observing that, 
in addition to the root label $(1,1)$ which contributes for $xuv$, 
labels of the form $(1,k)$ are obtained only as the first production of the first row of the rule~\eqref{rule:m_slicings}. 
In Equation~\eqref{eqn:skinny_Fi}, the first (resp. second) term accounts for labels of the form $(i,k)$ 
produced \emph{via} the second (resp. first) row of the rule~\eqref{rule:m_slicings}.
Finally, in Equation~\eqref{eqn:skinny_Fm}, the first (resp. second, third, fourth) term 
corresponds to the productions of the second row of the rule~\eqref{rule:m_slicings} for $h=m-1$ 
(resp. third row, fourth row, first row of the rule~\eqref{rule:m_slicings}). 
In writing this equation, especially for the second term, it is important to remember that 
the exponent of $u$ in $F_m(u,v)$ is not identically $m$, but takes all possible values starting from $m$. 

\smallskip

Lastly, we consider $m$-row-restricted slicings. 
As previously mentioned, $m=0$ leads to a trivial combinatorial class, while $m=1$ yields the Catalan numbers and their generating function $F_\text{Cat}(x)$ as per~\eqref{eqn:catalan_gf}. 

We thus fix some $m \geq 2$. The succession rule~\eqref{rule:row_restricted} yields a system of functional equations satisfied by the generating function for $m$-row-restricted slicings. 
More precisely, for any $i\leq m$, denote by $G_i(x;u,v)\equiv G_i(u,v) =\sum_{\alpha} x^{n(\alpha)} u^{h(\alpha)} v^{k(\alpha)}$ 
the trivariate generating function for $m$-row-restricted slicings 
whose label according to rule~\eqref{rule:row_restricted} is of the form $(i,\cdot)$.
Also in this case, for any $m\geq2$, the trivariate generating function for $m$-row-restricted slicings is given by $G_{m\text{-}RR}(x;u,v)\equiv G_{m\text{-}RR}(u,v)=\sum_i G_i(u,v)$.
Note that $G_i(u,v) = u^i G_i(1,v)$ for all $i \leq m$, which makes the variable $u$ unnecessary.
Rule~\eqref{rule:row_restricted} translates into the following system: 
\begin{align*}
G_1(u,v)&=xuv+xuv\left[G_1(1,v)+G_2(1,v)+\ldots+G_m(1,v)\right]\\[-5pt]
&\;\;\vdots\\[-5pt]
G_i(u,v)&=\frac{xu^iv}{1-v}\left[G_{i-1}(1,1)-G_{i-1}(1,v)\right]+xu^iv\left[G_i(1,v)+\ldots+G_m(1,v)\right] \quad \text{for } 1<i<m\\[-5pt]
&\;\;\vdots\\[-5pt]
G_m(u,v)&=\frac{xu^mv}{1-v}\left[G_m(1,1)-G_m(1,v)+G_{m-1}(1,1)-G_{m-1}(1,v)\right]+xu^mvG_m(1,v),
%\label{equation_fg_m-restr}
\end{align*}
or equivalently, written without $u$ in $H_i(v) \equiv G_i(1,v)$: 
\begin{align}
H_1(v)&=xv+xv\left[H_1(v)+H_2(v)+\ldots+H_m(v)\right] \tag{RR$_1$} \label{eqn:rr_H1} \\[-5pt]
&\;\;\vdots\notag\\[-5pt]
H_i(v)&=\frac{xv}{1-v}\left[H_{i-1}(1)-H_{i-1}(v)\right]+xv\left[H_i(v)+\ldots+H_m(v)\right] \quad \text{for } 1<i<m \tag{RR$_i$} \label{eqn:rr_Hi} \\[-5pt]
&\;\;\vdots\notag \\[-5pt]
H_m(v)&=\frac{xv}{1-v}\left[H_m(1)-H_m(v)+H_{m-1}(1)-H_{m-1}(v)\right]+xvH_m(v). \tag{RR$_m$} \label{eqn:rr_Hm}
%\label{equation_fg_m-restr_simpler}
\end{align}

\subsection{The special case of $0$-skinny and $2$-row-restricted slicings}

In this subsection we prove the following surprising result, for which we presently have no bijective explanation.

\begin{theorem}
The number of $2$-row-restricted slicings is equal to the number of $0$-skinny slicings, for any fixed size.
\label{thm:2-row-restricted}
\end{theorem}

We first solve the generating function for $2$-row-restricted slicings, and obtain the following.
\begin{theorem}
The generating function $H(x)$ for 2-row-restricted slicings %is algebraic of degree $3$ and
satisfies the functional equation 
\begin{equation}
H(x)= \frac{x (H(x)+1)}{1-x(H(x)+1)^2}. \tag{$\dagger$} \label{equation_fg_sloane}
 \end{equation}
\label{thm:GF_0skinny}
\end{theorem}

\begin{proof}
For $2$-row-restricted slicings, the succession rule is
\begin{equation}
\textrm{root labeled } (1,1) \textrm{\quad and \quad } (h,k) \rightsquigarrow 
\begin{cases}
(1, k + 1), \ldots , (h, k + 1), \\
(2,1) , (2,2) , \ldots ,  (2,k) \\
\end{cases}  \label{rule:2-restricted} \tag{$\Upsilon_2$}
\end{equation}
and the corresponding system of functional equations is
\begin{align}
\begin{split}
H_1(v)&=xv+xv(H_1(v)+H_2(v)) \\
H_2(v)&=\frac{xv}{1-v}(H_2(1)-H_2(v)+H_{1}(1)-H_{1}(v))+xvH_2(v). 
\end{split}
\tag{2-RR} \label{eqn:2rr_H2}
\end{align}
The quantity we wish to solve is 
the generating function for $2$-row-restricted slicings, given by $H(x)\equiv G_{2\text{-}RR}(x;1,1) = H_1(1) + H_2(1)$.
Canceling $H_1(v)$ between~\eqref{eqn:2rr_H2}, we arrive at
\begin{equation*}\label{eqn:2rr_kernel}
K(v)H_2(v) = \frac{xv}{1-v}\left(\frac{-xv}{1-xv}+H_1(1)+H_2(1)\right)
\end{equation*}
where
\[K(v) = 1-xv+\frac{xv}{1-v}+\frac{x^2v^2}{(1-v)(1-xv)}.\]
This equation is susceptible to the \emph{kernel method}~\cite{GFGT,BM}. The equation $K(v)=0$ is cubic in $v$, and one of the three roots has a power series expansion in $x$ (the other two are not analytic at $x=0$). 
Letting $\lambda(x) \equiv \lambda$ denote this root, we then have
\[H(x) = H_1(1) + H_2(1) = \frac{x\lambda}{1-x\lambda}.\]
It follows that $\lambda = \frac{H}{x(H+1)}$, and the condition $K(\lambda)=0$ rewrites as 
\begin{equation}
xH^3 + 2xH^2 + (2x-1)H+x=0 \ , \tag{$\dagger'$} \label{equation_fg_sloane_aux}
\end{equation}
or equivalently equation~\eqref{equation_fg_sloane}. 
\end{proof}

\begin{remark}
It follows that the sequence for 2-row-restricted slicings is (up to the first term) the same as sequence \textsc{a106228} in~\cite{OEIS}. 
Indeed, the generating function $S$ for sequence \textsc{a106228} is characterized by $xS^3-xS^2+(x-1)S+1=0$~\cite{C-machines}, 
and with \eqref{equation_fg_sloane_aux} it is immediate to check that $H+1$ satisfies this equation. 
\end{remark}

\begin{proof}[Proof of Theorem~\ref{thm:2-row-restricted}]
The generating function $F_{0\text{-}Sk}(u,v)$ for $0$-skinny slicings satisfies~\eqref{eqn:0skinny_func_eqn}, and this equation can also be solved via the kernel method. 
However, things are somewhat more complicated here, due to the presence of two catalytic variables. 
First, we rearrange the equation into the kernel form
\begin{equation*}
L(u,v)F_{0\text{-}Sk}(u,v) = xuv + xu\left(\frac{v}{1-u}-\frac{1}{1-v}\right)F_{0\text{-}Sk}(1,v) +\frac{xuv}{1-v}F_{0\text{-}Sk}(1,1)
\end{equation*}
where
\begin{equation*}\label{eqn:Luv}
L(u,v) = 1-xu + \frac{xuv}{1-u}.
\end{equation*}
The equation $L(u,v)=0$ is quadratic in $u$, and one of the two roots is a power series in $x$ with coefficients in $\mathbb Z[v]$ (the other is not analytic at $x=0$). 
We denote this root by 
\[
\mu(x,v) \equiv \mu(v) = \frac{1+x-xv - \sqrt{1-2 x-2xv+x^2-2 x^2v+x^2v^2} }{2x}.
\]
It follows that
\[M(v)F_{0\text{-}Sk}(1,v) = v + \frac{v}{1-v}F_{0\text{-}Sk}(1,1) \textrm{ \ where \ } M(v) = \frac{1}{1-v}-\frac{v}{1-\mu(v)}.\]
Now the kernel method can be applied \emph{again} -- the equation $M(v) = 0$ is (after rearrangement) quartic in $v$, 
namely, it is $4xv(1-v+xv-xv^2+xv^3)=0$. 
One of the three non-zero roots of this equation has a power series expansion in $x$. 
Denoting by $\kappa(x) \equiv \kappa$ this root, we finally have $F_{0\text{-}Sk}(1,1) = \kappa -1$. Some elementary manipulations in {\sc Mathematica} 
(or any other computer algebra system) show that $F_{0\text{-}Sk}(1,1)$ also satisfies~\eqref{equation_fg_sloane}.
\end{proof}

We point out that D. Callan indicates in~\cite{OEIS} that $F_{0\text{-}Sk}\equiv F_{0\text{-}Sk}(1,1)$ is also the generating function for Schr\"oder paths with no triple descents, 
\emph{i.e.}~having no occurrence of the factor $DDD$, where $D$ encodes the down step. 
It would be interesting to provide a bijection between Schr\"oder slicings and Schr\"oder paths 
whose restriction to $0$-skinny slicings yields a bijection with Schr\"oder paths having no triple descents. 
However, our first investigations in this direction have been unsuccessful. 

\begin{remark}
It does not hold in general that there are as many $m$-skinny slicings as $(m+2)$-row-restricted slicings: 
already for $m=1$, there are $91$ $3$-row-restricted slicings but $90$ Schr\"oder (\emph{i.e.}~$1$-skinny) slicings of size $5$. 
More precisely, out of the $92$ Baxter slicings of size $5$, only 
\tikz[scale=0.2]{
\draw (0,0) rectangle (4,2);
\draw (0,1) -- (4,1);
\draw (1,0) -- (1,1);
\draw (2,0) -- (2,1);
\draw (3,0) -- (3,1);
}
is not $3$-row-restricted, 
but both \tikz[scale=0.2]{
\draw (0,0) rectangle (2,2);
\draw (0,2) rectangle (3,3);
\draw (2,1) -- (3,1) -- (3,2);
\draw (1,0) -- (1,2);
\draw (0,1) -- (1,1);
}
and \tikz[scale=0.2]{
\draw (0,0) rectangle (2,2);
\draw (0,2) rectangle (3,3);
\draw (2,1) -- (3,1) -- (3,2);
\draw (0,1) -- (2,1);
\draw (1,0) -- (1,1);
}
are not Schr\"oder slicings.
\end{remark}

\subsection{Generating functions of $m$-skinny and $m$-row-restricted slicings for general $m$}

In this final subsection, we outline an approach for solving the generating functions for $m$-skinny and $m$-row-restricted slicings, for arbitrary $m$. While this method is \emph{provably} correct for small $m$, we do not know how to show that all of the steps always work, and so we omit any proofs. The following thus remains a conjecture.

\begin{conjecture}\label{conj:algebraic}
For all finite $m\geq0$, the generating functions for $m$-skinny and $m$-row-restricted slicings are algebraic.
\end{conjecture}

Table~\ref{table:conj} summarizes the cases for which we know that the above statement holds, 
either from previous results in this paper, or from the method described below. 

\begin{table}[ht]
\begin{center}
\begin{tabular}{|l|c|c|c|c|c|c|}
\hline
$m$ & $0$ & $1$ & $2$ & $3$ & $4$ & $5$ \\
\hline
$m$-row & $1/(1-x)$ & $F_\text{Cat}(x)$ & eq.~\eqref{equation_fg_sloane} & eq.~\eqref{eq:GF3RR} and & eq.~\eqref{eq:GF4RR}  & eq.~\eqref{eq:GF5RR}  \\
-restricted & \S\ref{paragraph:row-restricted} & \S\ref{paragraph:row-restricted} & Thm~\ref{thm:GF_0skinny} & eq.~\eqref{eq:eqGF3RR} p.\pageref{eq:eqGF3RR} &  p.\pageref{eq:GF4RR} & p.\pageref{eq:GF5RR} \\
\hline
$m$-skinny & eq.~\eqref{equation_fg_sloane}  & $F_\text{Sch}(x)$  & eq.~\eqref{eq:GF2sk} & eq.~\eqref{eq:GF3sk} &  & \\
 & Thm~\ref{thm:2-row-restricted} & Thm~\ref{thm:SchPFP} & p.\pageref{eq:GF2sk}  & p.\pageref{eq:GF3sk} &   &  \\
\hline
\end{tabular}
\end{center}
\caption{For small values of $m$, the statement of Conjecture~\ref{conj:algebraic} holds. 
Each cell of the table gives the corresponding generating function and/or an equation characterizing it.}
\label{table:conj}
\end{table}

\smallskip

We will mostly focus on $m$-row-restricted slicings, and briefly explain at the end how to modify the method to solve $m$-skinny slicings. In the following it is assumed that $m\geq 3$.

\medskip

\noindent \textbf{Step 1.}~Note that the system~\eqref{eqn:rr_H1}--\eqref{eqn:rr_Hm} can be rewritten in the form of a matrix equation
\begin{equation}
\mathbf K_m(v) \mathbf H_m(v) = \mathbf B_m(v) \mathbf H_m(1) + \mathbf C_m(v), \tag{Mat-RR}
\label{eqn:matrices}
\end{equation}
where
\[\mathbf H_m(v) = \begin{pmatrix} H_1(v) \\ \vdots \\ H_m(v) \end{pmatrix}, \qquad \mathbf K_m(v) = \begin{pmatrix} 1-xv & -xv & -xv & -xv & \cdots & -xv \\ \frac{xv}{1-v} & 1-xv & -xv & -xv &  \cdots & -xv \\ 0 & \frac{xv}{1-v} & 1-xv & -xv & \cdots & -xv \\ \vdots & \ddots & \ddots & \ddots & \ddots & \vdots \\ 0 & 0 & \cdots & \frac{xv}{1-v} & 1-xv & -xv \\ 0 & 0 & 0 & \cdots & \frac{xv}{1-v} & 1-xv+\frac{xv}{1-v} \end{pmatrix},\]
\[\mathbf B_m(v) = \begin{pmatrix} 0 & 0 & 0 & 0 & \cdots & 0 \\ \frac{xv}{1-v} & 0 & 0 & 0 & \cdots & 0 \\ 0 & \frac{xv}{1-v} & 0 & 0 & \cdots & 0 \\ 0 & 0 & \frac{xv}{1-v} & 0 & \cdots & 0 \\ \vdots & \vdots & \ddots & \ddots & \ddots & \vdots \\ 0 & 0 & 0 & \cdots & \frac{xv}{1-v} & \frac{xv}{1-v} \end{pmatrix}\qquad\text{and}\qquad \mathbf C_m(v) = \begin{pmatrix} xv \\ 0 \\ \vdots \\ 0\end{pmatrix}.\]

\noindent \textbf{Step 2.}~The determinant $|\mathbf K_m(v)|$ is a rational function of $x$ and $v$ which can be shown to be not identically zero for any $m$. It follows that, in general, $\mathbf K_m(v)$ has an inverse. Write $\mathbf K_m^*(v) = |\mathbf K_m(v)|\mathbf K_m^{-1}(v)$ (the transpose of the matrix of cofactors of $\mathbf K_m(v)$). It can further be shown that none of the elements of the last row of $\textbf K_m^*(v)$ are identically zero.

\medskip

\noindent \textbf{Step 3.}~Multiply~\eqref{eqn:matrices} on the left by $\mathbf K^*_m(v)$ to give
\begin{equation}
|\mathbf K_m(v)| \mathbf H_m(v) = \mathbf K^*_m(v) \left[\mathbf B_m(v) \mathbf H_m(1) + \mathbf C_m(v)\right].\tag{$\nabla$}
\label{eqn:new_system}
\end{equation}
This can be viewed as a system of $m$ kernel equations, where the kernel (namely $|\mathbf K_m(v)|$) is the same for each. The LHS of the $m$-th equation of~\eqref{eqn:new_system} is $|\mathbf K_m(v)| H_m(v)$, while the RHS is a linear combination of all the $m$ unknowns $H_1(1),\ldots,H_m(1)$. Furthermore, note that in~\eqref{eqn:rr_H1}--\eqref{eqn:rr_Hm}, the unknowns $H_{m-1}(1)$ and $H_m(1)$ only appear together as $H_{m-1}(1)+H_m(1)$. Writing this latter quantity as $H_{(m-1)+m}(1)$, we now see that there are really only $m-1$ unknowns on the RHS of~\eqref{eqn:new_system}. 

\medskip

\noindent \textbf{Step 4.}~The equation $|\mathbf K_m(v)| = 0$ can be shown to have precisely $m-2$ roots (in the variable $v$) which are Puiseux series in $x$. Denote these roots by $\nu_1(x),\ldots,\nu_{m-2}(x)$.

\medskip

\noindent \textbf{Step 5.}~Substitute $v=\nu_i(x)$ into the first of the $m$ equations comprising the system~\eqref{eqn:new_system}, for $i=1,\dots,m-2$. This yields a system of $m-2$ linear equations in $m-1$ unknowns. 

\medskip

\noindent \textbf{Step 6.}~To obtain one more equation, set $v=1$ in~\eqref{eqn:rr_H1} (again combining $H_{m-1}(1)+H_m(1)$ as $H_{(m-1)+m}(1)$).

\medskip 

\noindent \textbf{Step 7.}~Solve this entire linear system of $m-1$ equations with $m-1$ unknowns, and add all solutions together to obtain the generating function $H(x)$ for $m$-row-restricted slicings.

\medskip

It is the validity of Step 7 which we are unable to verify in general. To do so, it would be necessary to show that the $\nu_i(x)$ are distinct and linearly independent functions of $x$, and moreover that the $(m-1)$-th equation obtained in Step 6 is independent of those obtained in Step 5. Nevertheless, this method has been verified manually for $m\leq 5$.

The series expansion of the generating function for $3$-row-restricted slicings is
\begin{equation}
x + 2x^2 + 6x^3 + 22x^4 + 91x^5 + 405x^6 + 1893x^7 + 9163x^8 + 45531x^9 + 230902x^{10} + O(x^{11}).\tag{GF$_\text{3RR}$}\label{eq:GF3RR}  
\end{equation}
With some help from {\sc Mathematica}, and here specifically from M.~Kauers' ``Guess" package, one finds that this generating function is a root of the cubic polynomial
\[x + 2 x^2 + x^3 + (-1 - 2 x + 2 x^2 + 3 x^3) H + (2 - 2 x^2 + 3 x^3) H^2 + (-1 + 3 x - 2 x^2 + x^3) H^3.\tag{Alg$_\text{3RR}$}\label{eq:eqGF3RR}  \]
The generating functions for $m=4$ and $m=5$ have the respective series expansions
\begin{align} 
&x + 2x^2 + 6x^3 + 22x^4 + 92x^5 + 421x^6 + 2051x^7 + 10449x^8 +55023x^9 + 297139x^{10} + O(x^{11}) \tag{GF$_\text{4RR}$}\label{eq:GF4RR}\\
&x + 2x^2 + 6x^3 + 22x^4 + 92x^5 + 422x^6 + 2073x^7 + 10724x^8 + 57716x^9 + 320312 x^{10} + O(x^{11}) \tag{GF$_\text{5RR}$}\label{eq:GF5RR}.
\end{align}
By construction these functions must be algebraic, but as the order of the kernel equation $|\mathbf K_m(v)|=0$ increases with $m$, we have been unable to determine precisely the polynomials satisfied by these generating functions.

\medskip

We now briefly turn to $m$-skinny slicings. The method is largely the same, with some minor differences. Firstly, an additional step is required at the start.

\medskip

\noindent \textbf{Step 0$^*$.}~Substitute $u=\mu(v)$ into~\eqref{eqn:skinny_Fm}, where $\mu(v)$ is the power series root of $L(u,v)$ as defined in the proof of Theorem~\ref{thm:2-row-restricted}.
This eliminates the term $F_m(u,v)$, leaving an equation relating $F_{m-1}(1,1), F_{m-1}(1,v), F_m(1,1)$ and $F_m(1,v)$.
Meanwhile, the variable $u$ is unnecessary in equations~\eqref{eqn:skinny_Fi} for $1\leq i<m$, so set it to 1.

\medskip

The remaining steps can then be adapted to this system of equations, with $F_i(1,v)$ taking the place of $H_i(v)$. One key difference is that $F_{m-1}(1,1)$ and $F_m(1,1)$ cannot be combined, so there are $m$ unknowns that need to be solved instead of $m-1$. However, this time the kernel (again the determinant of a matrix) has $m-1$ Puiseux series roots instead of $m-2$, which exactly compensates for this problem.

When $m=2$ the desired solution $F_1(1,1)+F_2(1,1)$ enumerating $2$-skinny slicings has the form
\begin{equation}
x+2 x^2+6 x^3+22 x^4+92 x^5+419 x^6+2022 x^7+10168 x^8+52718 x^9+279820 x^{10} + O(x^{11}). \tag{GF$_\text{2Sk}$}\label{eq:GF2sk}
\end{equation}
This generating function is a root of the quintic polynomial
\[x^3-x^2 (1-6 x) F-3 x^2 (2-5 x) F^2+x (2-13 x+19 x^2) F^3+x (5-12 x+12 x^2) F^4-(1-3 x+4 x^2-3 x^3) F^5.\]

When $m=3$ the generating function for $3$-skinny slicings has the form
\begin{equation}
x+2 x^2+6 x^3+22 x^4+92 x^5+422 x^6+2070 x^7+10668 x^8+57061 x^9+314061 x^{10} + O(x^{11}). \tag{GF$_\text{3Sk}$}\label{eq:GF3sk}
\end{equation}
By construction it is certainly algebraic, but we make no attempt here to write down the polynomial of which it is a root.

\paragraph*{Acknowledgments}
We are grateful to anonymous referees whose comments, at various stages of the writing 
(including the short version~\cite{FPSAC}),
were very helpful to improve the clarity of the current presentation. 

Many thanks also to Mireille Bousquet-M\'elou for patiently sharing with us her expertise with the kernel method (in theory and in \textsc{Maple}!).

The first author was supported by the Pacific Institute for the Mathematical Sciences and in particular the Collaborative Research Group in Applied Combinatorics, and the Australian Research Council grant DE170100186.

\end{document}